\documentclass[12pt,a4paper]{article}
\usepackage{amssymb}
\usepackage{amsmath, amscd}
\usepackage{amsthm}
\usepackage{mathtools}
\usepackage{mathrsfs}
\usepackage{booktabs}
\usepackage{perpage}
\usepackage[all,cmtip]{xy}
\usepackage{setspace}
\usepackage{amsbsy}
\usepackage[T1]{fontenc}
\usepackage{verbatim}
\usepackage{mathdots}
\usepackage{mathrsfs}
\usepackage{ifthen}
\usepackage{blkarray}
\usepackage{eucal}
\usepackage{multirow}
\usepackage{enumerate}
\usepackage[dvipsnames]{xcolor}
\usepackage{tikz}
\usepackage{fullpage}
\usepackage{colonequals}
\usepackage{braket}
\usepackage{lipsum}
\usepackage{hyperref}
\usepackage{float}

\newcommand{\relmiddle}[1]{\mathrel{}\middle#1\mathrel{}}
\mathtoolsset{showonlyrefs=true}

\numberwithin{equation}{subsection}

\newtheorem{theorem}{Theorem}[subsection]

\newtheorem{lemma}[theorem]{Lemma}
\newtheorem{conjecture}[theorem]{Conjecture}

\newtheorem{corollary}[theorem]{Corollary}
\newtheorem{definition}[theorem]{Definition}

\newtheorem{proposition}[theorem]{Proposition}

\newtheorem{assumption}[theorem]{Assumption}

\newtheorem*{thma}{Theorem A}

\newtheorem*{thm1}{Theorem 1}

\newtheorem*{thmb}{Theorem B}
\newtheorem*{thmc}{Theorem C}
\newtheorem*{thmd}{Theorem D}

\newtheorem*{conj1}{Conjecture 1}

\newtheorem*{ques1}{Question 1}

\theoremstyle{remark}

\newtheorem{rmk}[theorem]{Remark}
\newtheorem{exa}[theorem]{Example}

\numberwithin{equation}{subsection}

\newcommand{\GZip}{\mathop{\text{$G$-{\tt Zip}}}\nolimits}

\newcommand{\GoneZip}{\mathop{\text{$G_1$-{\tt Zip}}}\nolimits}
\newcommand{\GtwoZip}{\mathop{\text{$G_2$-{\tt Zip}}}\nolimits}

\newcommand{\GF}{\mathop{\text{$G$-{\tt ZipFlag}}}\nolimits}

\DeclareMathOperator{\pr}{pr}
\DeclareMathOperator{\Gal}{Gal}

\DeclareMathOperator{\Sbt}{Sbt}

\DeclareMathOperator{\GL}{GL}

\newskip\procskipamount
\newskip\interskipamount
\newskip\refskipamount
\newcommand{\procskip}{\vskip\procskipamount}
\newcommand{\interskip}{\vskip\interskipamount}
\newcommand{\refskip}{\vskip\refskipamount}

\newcommand{\procbreak}{\par
   \ifdim\lastskip<\procskipamount\removelastskip
   \penalty-100
   \procskip\fi
   \noindent\ignorespaces}

\newcommand{\titlebreak}{\par%
\ifdim\lastskip<\interskipamount\removelastskip%
\penalty10000%
\interskip\fi%
\noindent}%

\newcommand{\interbreak}{\par%
\ifdim\lastskip<\interskipamount\removelastskip%
\penalty-100%
\interskip\fi%
\noindent\ignorespaces}%

\newcommand{\refbreak}{\par%
\ifdim\lastskip<\refskipamount\removelastskip%
\penalty-100%
\refskip\fi%
\noindent\ignorespaces}%

\newcounter{listcounter}
\newcounter{deflistcounter}
\newcounter{equivcounter}

\newskip{\itemsepamount}
\newskip{\topsepamount}
\newenvironment{assertionlist}{%
  \begin{list}
    {\upshape (\arabic{listcounter})}
    {\setlength{\leftmargin}{18pt}
     \setlength{\rightmargin}{0pt}
     \setlength{\itemindent}{0pt}
     \setlength{\labelsep}{5pt}
     \setlength{\labelwidth}{13pt}
     \setlength{\listparindent}{\parindent}
     \setlength{\parsep}{0pt}
     \setlength{\itemsep}{\itemsepamount}
     \setlength{\topsep}{\topsepamount}
     \usecounter{listcounter}}}
  {\end{list}}

\newenvironment{definitionlist}{%
  \begin{list}
    {\upshape (\alph{deflistcounter})}
    {\setlength{\leftmargin}{18pt}
     \setlength{\rightmargin}{0pt}
     \setlength{\itemindent}{0pt}
     \setlength{\labelsep}{5pt}
     \setlength{\labelwidth}{13pt}
     \setlength{\listparindent}{\parindent}
     \setlength{\parsep}{0pt}
     \setlength{\itemsep}{\itemsepamount}
     \setlength{\topsep}{\topsepamount}
     \usecounter{deflistcounter}}}
  {\end{list}}

\newenvironment{equivlist}{%
  \begin{list}
    {\upshape (\roman{equivcounter})}
    {\setlength{\leftmargin}{18pt}
     \setlength{\rightmargin}{0pt}
     \setlength{\itemindent}{0pt}
     \setlength{\labelsep}{5pt}
     \setlength{\labelwidth}{13pt}
     \setlength{\listparindent}{\parindent}
     \setlength{\parsep}{0pt}
     \setlength{\itemsep}{\itemsepamount}
     \setlength{\topsep}{\topsepamount}
     \usecounter{equivcounter}}}
  {\end{list}}

\newenvironment{bulletlist}{%
  \begin{list}
    {\upshape \textbullet}
    {\setlength{\leftmargin}{18pt}
     \setlength{\rightmargin}{0pt}
     \setlength{\itemindent}{0pt}
     \setlength{\labelsep}{6pt}
     \setlength{\labelwidth}{12pt}
     \setlength{\listparindent}{\parindent}
     \setlength{\parsep}{0pt}
     \setlength{\itemsep}{\itemsepamount}
     \setlength{\topsep}{\topsepamount}}}
  {\end{list}}

% ...
\newcommand{\Acal}{{\mathcal A}}

\newcommand{\Ccal}{{\mathcal C}}

\newcommand{\Fcal}{{\mathcal F}}
\newcommand{\Gcal}{{\mathcal G}}

\newcommand{\Lcal}{{\mathcal L}}

\newcommand{\Ocal}{{\mathcal O}}
\newcommand{\Pcal}{{\mathcal P}}

\newcommand{\Scal}{{\mathcal S}}

\newcommand{\Ucal}{{\mathcal U}}
\newcommand{\Vcal}{{\mathcal V}}

\newcommand{\Xcal}{{\mathcal X}}

\newcommand{\pfr}{{\mathfrak p}}

\renewcommand{\AA}{\mathbb{A}}

\newcommand{\CC}{\mathbb{C}}

\newcommand{\FF}{\mathbb{F}}
\newcommand{\GG}{\mathbb{G}}

\newcommand{\KK}{\mathbb{K}}

\newcommand{\NN}{\mathbb{N}}

\newcommand{\PP}{\mathbb{P}}
\newcommand{\QQ}{\mathbb{Q}}
\newcommand{\RR}{\mathbb{R}}

\newcommand{\ZZ}{\mathbb{Z}}

\newcommand{\Ascr}{{\mathscr A}}

\newcommand{\Hom}{{\rm Hom}}

\newcommand{\loccit}{{\em loc.\ cit. }}

\newcommand{\Sh}{{\rm Sh }}

\newcommand{\tor}{{\rm tor}}

\DeclareMathOperator{\Res}{Res}

\DeclareMathOperator{\Adj}{Adj}
\DeclareMathOperator{\GS}{GS}
\DeclareMathOperator{\flag}{flag}
\DeclareMathOperator{\Flag}{Flag}
\DeclareMathOperator{\zip}{zip}
\DeclareMathOperator{\Adm}{Adm}
\DeclareMathOperator{\PGL}{PGL}
\DeclareMathOperator{\GU}{GU}
\DeclareMathOperator{\U}{U}

\DeclareMathOperator{\GSp}{GSp}
\DeclareMathOperator{\Norm}{Norm}
\DeclareMathOperator{\Ind}{Ind}

\newcommand{\gen}{\mathsf{gen}}
\newcommand{\Ha}{\mathsf{Ha}}
\newcommand{\ha}{\mathsf{ha}}
\newcommand{\pha}{\mathsf{pHa}}
\newcommand{\zipsf}{\mathsf{zip}}
\newcommand{\GSsf}{\mathsf{GS}}
\newcommand{\lwsf}{\mathsf{lw}}
\newcommand{\lw}{\mathsf{lw}}

\newcommand{\hw}{\mathsf{hw}}
\newcommand{\low}{\mathsf{low}}
\newcommand{\high}{\mathsf{high}}

\author{Jean-Stefan Koskivirta}

\begin{document}

\title{Cohomology vanishing for Ekedahl--Oort strata on Hilbert--Blumenthal Shimura varieties}

\date{}
\maketitle

\begin{abstract}
We prove vanishing results for the $0$th cohomology of automorphic line bundles on the Ekedahl--Oort strata of the special fiber of Hilbert--Blumenthal Shimura varieties. This work vastly extends to all Ekedahl--Oort strata previous results of Diamond--Kassaei and Goldring--Koskivirta. Furthermore, our results also cover the case of unitary Shimura varieties of rank $\leq 2$ attached to an arbitrary CM extension. We also prove a conjecture of Goldring--Koskivirta on the set of possible weights of automorphic forms.
\end{abstract}

\section{Introduction}

In this paper, we prove a sharp vanishing result for automorphic line bundles on certain stratifications of Shimura varieties. The family of Shimura varieties considered in this paper will be called "of $A_1$-type". These include Hilbert--Blumenthal Shimura varieties and unitary Shimura varieties attached to a hermitian space of dimension $2$ over a CM extension $\mathbf{E}/\mathbf{F}$. We only consider the cohomology in degree $0$. We start be reviewing the vanishing results of Diamond--Kassaei (\cite{Diamond-Kassaei-cone-minimal}) and Goldring--Koskivirta (\cite{Goldring-Koskivirta-global-sections-compositio}), of which this paper is a generalization. Fix a totally real extension $\mathbf{F}/\QQ$. Hilbert--Blumenthal Shimura varieties are attached to a certain reductive $\QQ$-group $\mathbf{G}\subset \Res_{\mathbf{F}/\QQ}(\GL_{2,\mathbf{F}})$ (see section \ref{Shim-A1-type-sec} for details). They parametrizes abelian varieties of rank $n=[\mathbf{F}:\QQ]$ endowed with a principal polarization, a compatible action of $\Ocal_{\mathbf{F}}$ and a level structure. Let $X=X_{\overline{\FF}_p}$ denote the special fiber (over $\overline{\FF}_p$) of a Hilbert--Blumenthal Shimura variety attached to $\mathbf{F}$ at a prime number $p$ of good reduction. It is a quasi-projective variety of dimension $n$ defined over $\FF_p$. Denote by $\Sigma$ the set of embeddings $\tau \colon \mathbf{F}\to \overline{\QQ}_p$. Since by assumption $\mathbf{F}$ is unramified at $p$, there is a natural action of the Frobenius homomorphism $\sigma\in \Gal(\QQ_p^{\rm ur}/\QQ_p)=\Gal(\overline{\FF}_p/\FF_p)$ on $\Sigma$ that we denote by $\tau\mapsto \sigma\tau$. Denote by $\Omega$ the Hodge vector bundle of $X$, namely the vector bundle $e^*(\Omega^1_{\Ascr/X})$ where $\Ascr/X$ is the universal abelian variety and $e\colon X\to \Ascr$ is the unit section. The action of $\Ocal_{\mathbf{F}}$ on $\Ascr$ decomposes $\Omega$ as a direct sum $\Omega=\bigoplus_{\tau\in \Sigma} \omega_\tau$. For a tuple $\mathbf{k}=(k_\tau)_{\tau\in \Sigma}\in \ZZ^\Sigma$, we define a line bundle
\begin{equation}\label{omegak-intro}
    \omega^{\mathbf{k}} \colonequals \bigotimes_{\tau\in \Sigma} \omega^{-k_\tau}_\tau.
\end{equation}
Note the minus sign in the exponent of $\omega_\tau$; this sign convention differs from \cite{Diamond-Kassaei-cone-minimal}. We use it for consistency with the results of \cite{Goldring-Koskivirta-global-sections-compositio}, which are formulated for more general groups. The global sections $H^0(X,\omega^{\mathbf{k}})$ are called Hilbert modular forms of weight $\mathbf{k}$. The Ekedahl--Oort (EO) stratification of $X$ is defined as follows: Two points $x,y\in X$ are in the same stratum if the abelian varieties $A_x$, $A_y$ corresponding to $x,y$ satisfy that $A_x[p]\simeq A_y[p]$, where we require this isomorphism to be compatible with the polarizations and the action of $\Ocal_{\mathbf{F}}$. There is a unique open stratum called the ordinary locus, where $A_x$ is an ordinary abelian variety. The EO strata are parametrized by subsets $S\subset \Sigma$, and the dimension of the stratum $X_S$ corresponding to the subset $S$ is $|S|$. For each $\tau\in \Sigma$, Andreatta--Goren (\cite{Andreatta-Goren-book}) constructed a partial Hasse invariant $\Ha_\tau$ whose vanishing locus is precisely the codimension one stratum corresponding to the subset $S\setminus\{\tau\}$. The weight of $\Ha_\tau$ is given by $\ha_\tau \colonequals e_\tau - p e_{\sigma^{-1} \tau}$. It was proved by Goldring and the author (\cite[Theorem D (a)]{Goldring-Koskivirta-global-sections-compositio}), as well as Diamond--Kassaei (\cite[Corollary 8.2]{Diamond-Kassaei-cone-minimal}) that for any $\mathbf{k}\in \ZZ^\Sigma$ which is outside of the cone spanned (over $\QQ_{\geq 0}$) by the weights $\ha_\tau$ that $H^0(X,\omega^{\mathbf{k}})=0$. Furthermore, \cite{Goldring-Koskivirta-global-sections-compositio} shows a more general result for other Ekedahl--Oort strata, termed admissible. The admissibility is an explicit combinatorial condition on the subset $S$ (see \loccit Definition 4.2.1). For any $S$ and any $\tau\in \Sigma$, one can define a generalized partial Hasse invariant $\Ha_{S,\tau}$ over $\overline{X}_S$ (the Zariski closure of $X_S$ endowed with the reduced structure). For $\tau\in S$, the section $\Ha_{S,\tau}$ is simply the restriction of $\Ha_\tau$ to $\overline{X}_S$. For $\tau\notin S$, the section $\Ha_{S,\tau}$ is a non-vanishing section of weight $\ha_{S,\tau}\colonequals -e_\tau - p e_{\sigma^{-1} \tau}$. Denote by $\Ccal_{\pha,S}$ the elements of $\ZZ^\Sigma$ which can be spanned (over $\QQ_{\geq 0}$) by the weights $\ha_{S,\tau}$. Then, Goldring and the author proved:
\begin{thm1}[{\cite[Theorem 4.2.3]{Goldring-Koskivirta-global-sections-compositio}}]
Let $S$ be an admissible subset. Then for any weight $\mathbf{k}\notin \Ccal_{\pha,S}$, the space $H^0(\overline{X}_S,\omega^{\mathbf{k}})$ is zero.
\end{thm1}
The maximal stratum and the one-dimensional strata are always admissible. When $p$ splits completely in $\mathbf{F}$, all EO strata are admissible. On the other hand, when $p$ is inert in $\mathbf{F}$, very few strata are admissible. In this article, we give a vanishing result that applies to all Ekedahl--Oort strata and generalizes the above theorem. Moreover, we include other Shimura varieties attached to similar reductive groups.

We say that a reductive group $\mathbf{G}$ over a field $k$ is of $A_1$-type and rank $n$ if $\mathbf{G}^{\rm ad}_{\overline{k}}$ is isomorphic to a product of $n$ copies of $\PGL_{2,\overline{k}}$, where $\overline{k}$ is an algebraic closure of $k$. We let $X=X_{\overline{\FF}_p}$ be the good reduction special fiber of a Hodge-type Shimura variety attached to a group of $A_1$-type. Aside the case of Hilbert--Blumenthal varieties previously mentioned, we may also consider certain unitary Shimura varieties. Specifically, let $\mathbf{E}/\mathbf{F}$ be a CM-extension where $n=[\mathbf{F}:\QQ]$ and $(\mathbf{V},\psi)$ a hermitian space where $V$ is a 2-dimensional $\mathbf{E}$-vector space. Then, the group of unitary similitudes $\mathbf{G}\colonequals \GU(\mathbf{V},\psi)$ (restricted to $\QQ$) is a group of $A_1$-type of rank $n$. Contrary to the Hilbert--Blumenthal case, the Hodge parabolic $P$ (the stabilizer of the Hodge filtration) may be larger than a Borel subgroup in the unitary case. Let $\Sigma$ denote the set of embeddings $\tau \colon \mathbf{F}\to \RR$ and let $\Sigma_0 \subset \Sigma$ be the subset of embeddings where $\mathbf{V}\otimes_{\mathbf{F},\tau} \RR$ has signature $(2,0)$ or $(0,2)$. Then, the Hodge parabolic $P$ (viewed in $\mathbf{G}^{\rm ad}$) will contain all the factors corresponding to elements of $\Sigma_0$. Fix an isomorphism $\CC\simeq \overline{\QQ}_p$ and view $\Sigma$ as a subset of $\Hom_{\QQ}(\mathbf{F},\overline{\QQ}_p)$. Denote again by $\sigma$ the action of Frobenius on $\Sigma$. Then, $\Sigma$ decomposes as a disjoint union of orbits under the action of $\sigma$. All of our results can be reduced to the case of a single orbit. The reason is that we use as a main tool the stack of $G$-zips of Pink--Wedhorn--Ziegler (\cite{Pink-Wedhorn-Ziegler-zip-data,Pink-Wedhorn-Ziegler-F-Zips-additional-structure}), where $G$ is the $\overline{\FF}_p$ group obtained by reducing $\mathbf{G}$ modulo $p$, and this stack decomposes as a direct product with respect to the $\sigma$-orbits in $\Sigma$. Therefore, for the rest of this introduction, we will assume that $p$ is an inert prime in $\mathbf{F}$.

We order the elements of $\Sigma$ as $\tau_1,\dots, \tau_n$ such that $\tau_{i+1}=\sigma \tau_i$, where the index $i$ is taken modulo $n$. We identify in this way $\Sigma$ with $E_n=\{1,\dots,n\}$. Furthermore, we write $R\subset E_n$ for the subset corresponding to $\Sigma_0$ and call it the parabolic type of $X$. For example, when $X$ is a Hilbert--Blumenthal Shimura variety, we have $R=\emptyset$. We let $\Flag(X)$ denote the flag space of $X$, as defined by Goldring and the author in \cite{Goldring-Koskivirta-Strata-Hasse}, based on previous work by Ekedahl--van der Geer (\cite{Ekedahl-Geer-EO}). It parametrizes pairs $(x,\Fcal_{\bullet})$ where $x$ is a point of $X$ and $\Fcal_\bullet$ is a full flag in $H^1_{\rm dR}(A_x)$ refining the Hodge filtration. The natural projection $\pi\colon \Flag(X)\to X$ is proper with fibers isomorphic to a product of $|R|$ copies of $\PP^1$. When $R=\emptyset$, we simply have $\Flag(X)=X$. The flag space is endowed with a natural stratification $(\Flag(X)_S)_S$ parametrized by subsets $S\subset E_n$. Furthermore, for any $\mathbf{k}=(k_1,\dots, k_n)\in \ZZ$, it carries a line bundle $\Lcal(\mathbf{k})$ (for $R=\emptyset$, one has $\Lcal(\mathbf{k})=\omega^{\mathbf{k}}$). The push-forward $\pi_*(\Lcal(\mathbf{k}))$ is the automorphic vector bundle $\Vcal(\mathbf{k})$ attached to the induced representation $V(\mathbf{k})\colonequals \Ind_{B}^P(\mathbf{k})$. Denote by $\overline{\Flag}(X)_S$ the Zariski closure of $\Flag(X)_S$, endowed with the reduced structure. We are interested in the following problem:
\begin{ques1}
    For which $\mathbf{k}\in \ZZ^n$ is the space $H^0(\overline{\Flag}(X)_S,\Lcal(\mathbf{k}))$ nonzero?
\end{ques1}
In the Hilbert--Blumenthal case, $\Flag(X)_S$ is simply the Ekedahl--Oort stratum $X_S$, and Question 1 is thus very natural in and of itself. When $R\neq \emptyset$, we are mostly interested in the case of the maximal stratum, corresponding to $S=E_n$. Since $\pi_*(\Lcal(\mathbf{k}))=\Vcal(\mathbf{k})$, Question 1 boils down to understanding for which $\mathbf{k}\in \ZZ$ the space $H^0(X,\Vcal(\mathbf{k}))$ is nonzero. In the case $R\neq \emptyset$, even though we are interested in the case $S=E_n$, we need to solve Question 1 for each stratum $S$ in order to get to the maximal stratum, by an inductive procedure.

We now explain the results of this paper in details. We fix a subset $R\subset E_n$ and a Shimura variety $X=X_R$ of $A_1$-type and parabolic type $R$. We say that a subset $\Ccal\subset \ZZ^n$ is a $p$-cone if it is defined by finitely many inequalities of the form 
\begin{equation}\label{pexpr}\tag{1}
    \sum_{i=0}^{n-1} p^i \varepsilon_{i+d} \ x_{i+d} \leq 0, \quad (x_1,\dots, x_n)\in \ZZ^n
\end{equation}
where $\varepsilon_1,\dots, \varepsilon_n\in \{\pm 1\}$ and the index $i+d$ is taken modulo $n$. We call the expression appearing in \eqref{pexpr} a $p$-expression with starting index $d$. Such an expression is uniquely determined by its starting index $d$ and the set $T\subset E_n$ of indices $i$ with $\varepsilon_i=-1$. Let $S\subset E_n$ be a subset. We say that a $p$-cone $\Ccal$ is $S$-adapted if it is defined by exactly $|S|$ inequalities of the type \eqref{pexpr}, with starting index each of the element of $S$. An $S$-adapted $p$-cone $\Ccal$ is thus uniquely determined by a function
\begin{equation}
    \rho_\Ccal\colon S\to \Pcal(E_n)
\end{equation}
(where $\Pcal(E_n)$ is the powerset of $E_n$), which attaches to each element $s\in S$ the set of indices $i\in E_n$ such that $\varepsilon_i=-1$ in the $p$-expression with starting index $s$ defining $\Ccal$. We say that an $S$-adapted $p$-cone is homogeneous if $\rho_\Ccal$ is a constant function. For each subset $S\subset E_n$, there are $n-|S|$ non-vanishing sections $\Ha^{(i)}_{R,S}$ on the stratum $\overline{\Flag}(X)_S$ parametrized by the elements $i\in E_n\setminus S$ (see section \ref{subsec-admiss} for details). We write $\ha_{R,S}^{(i)}$ for the weight of $\Ha_{R,S}^{(i)}$. For any $\mathbf{k}$ in the subgroup spanned (over $\ZZ$) by these weights, the line bundle $\Lcal(\mathbf{k})$ is trivial on $\overline{\Flag}(X)_S$. We say that an $S$-adapted $p$-cone is admissible if the $p$-expressions that define $\Ccal$ all vanish on this subgroup (the term "admissible" is unrelated to the one used in Theorem 1).

Finally, we say that $\Ccal$ is positive if the leading coefficients of the $p$-expressions defining $\Ccal$ with starting index $s\in S$ is $1$ if and only if $s\in R$. Intuitively, this condition is imposed by the geometry of Shimura varieties: when $p$ tends to infinity, $\Ccal$ should contain the line bundles $\Lcal(\mathbf{k})$ that are ample in characteristic zero. We then have the following:

\begin{thma}
For any subset $S\subset E_n$, there exists a unique positive admissible homogeneous $S$-adapted $p$-cone.
\end{thma}

The inequalities defining the cone $\Ccal_S$ can be very explicitly determined using the notion of chain diagrams introduced in section \ref{subsec-chain}. We explain the link with Question 1. Write $\Ccal_S$ for the unique positive admissible homogeneous $S$-adapted $p$-cone afforded by Theorem A. The following is the main result of this paper:
\begin{thmb}
Let $S\subset E_n$ be a subset and $\mathbf{k}\in \ZZ^n$. If $\mathbf{k}\notin \Ccal_S$, then one has
\[
H^0(\overline{\Flag}(X)_S,\Lcal(\mathbf{k})) =0.
\]
\end{thmb}
When $R=\emptyset$, this result extends to all strata Theorem 1 of \cite{Goldring-Koskivirta-global-sections-compositio} explained earlier and gives a vanishing result for the cohomology of automorphic line bundles for any EO stratum of Hilbert--Blumenthal Shimura varieties. Note that in general, the natural cone that controls the vanishing of cohomology is not $\Ccal_{\pha,S}$ but rather $\Ccal_S$. This explains why Theorem 1 could only cover certain strata, precisely the ones for which $\Ccal_{\pha,S}$ and $\Ccal_S$ coincide. We explain how to determine the cone $\Ccal_S$ explicitly. For simplicity, we restrict here to the case of Hilbert--Blumenthal Shimura varieties (i.e. $R=\emptyset$). For any subset $S\subset E_n$, we define a subset $\Phi(S)\subset E_n$ as follows: For each element $s\in S$, we consider the sequence $s-1, s-2, \dots$ where these elements are taken modulo $n$. We denote by $\gamma(s)$ the smallest positive integer such that $s-\gamma(s)\in S$. Then, we define $\Phi(S)$ as the set
\begin{equation}
    \Phi(S)=\{s-i \ | \ s\in S, \ i \textrm{ odd}, \ 1\leq i <\gamma(s) \}.
\end{equation}
The cone $\Ccal_S$ is the $S$-adapted homogeneous $p$-cone that corresponds to the constant function $\rho\colon S\to \Pcal(E_n)$ with value $\Phi(S)$. For example, if $n=7$ and $S=\{1, 3\}$, then $\Phi(S)=\{5,7\}$. It follows that the cone $\Ccal_S$ is defined by the two inequalities
\[
\begin{cases}
    x_1 + p x_2 + p^2 x_3 + p^3 x_4 - p^4 x_5 + p^5 x_6 - p^6 x_7 \leq 0 \\
    p^5 x_1 + p^6 x_2 + x_3 + p x_4 - p^2 x_5 + p^3 x_6 - p^4 x_7 \leq 0.    
\end{cases}
\]
We briefly explain our approach to prove Theorem B. It is based on the notion of intersection-sum cone, first introduced by Goldring and the author in the recent preprints \cite{Goldring-Koskivirta-divisibility, Goldring-Koskivirta-GS-cone}. It is a natural construction that attaches recursively to each stratum $S$ a cone $\Ccal^{\cap,+}_S\subset \ZZ^n$. When $|S|=1$, it is simply the cone $\Ccal_{\pha,S}$ of partial Hasse invariants on $S$. For other strata $S$, it is defined as the convex hull of $\Ccal_{\pha,S}$ and the intersection of the cones $\Ccal^{\cap,+}_{S\setminus\{s\}}$ when $s$ varies in $S$. The main property of these cones is that the space $H^0(\overline{\Flag}(X)_S,\Lcal(\mathbf{k}))$ is always zero when $\mathbf{k}\notin \Ccal^{\cap,+}_S$. Theorem B is then a consequence of the following:

\begin{thmc}
    For any subset $S\subset E_n$, the cone $\Ccal^{\cap,+}_S$ coincides with the cone $\Ccal_S$ afforded by Theorem A.
\end{thmc}
This result illustrates that the formation of the intersection-sum cones is natural and contains meaningful information. Finally, we describe our last result, which is related to a conjecture of Goldring and the author, formulated initially in \cite{Goldring-Koskivirta-global-sections-compositio}. Define the set
\begin{equation}
    C_{X}\colonequals \{ \mathbf{k}\in \ZZ^n \ | \ H^0(X,\Vcal(\mathbf{k}))\neq 0\}.
\end{equation}
Since $\pi_*(\Lcal(\mathbf{k}))=\Vcal(\mathbf{k})$, this set coincides with the set of $\mathbf{k}\in \ZZ^n$ such that $\Lcal(\mathbf{k})$ admits nonzero sections on $\Flag(X)$. Define also $\Ccal_X$ as the saturation of $C_X$, i.e. the set of $\mathbf{k}\in \ZZ^n$ that are spanned (over $\QQ_{\geq 0}$) by the elements of $C_X$. The Cone Conjecture asserts that the set $\Ccal_X$ is encoded by the stack of $G$-zips. Zhang constructed in \cite{Zhang-EO-Hodge} a natural smooth morphism of stacks $\zeta\colon X\to \GZip^\mu$ which is surjective. For any $\mathbf{k}\in \ZZ^n$, the vector bundle $\Vcal(\mathbf{k})$ can also be defined on $\GZip^\mu$. We define similarly a set $C_{\zip}\subset \ZZ^n$ as the set of $\mathbf{k}\in \ZZ^n$ such that $H^0(\GZip^\mu,\Vcal(\mathbf{k}))\neq 0$ and define $\Ccal_{\zip}$ as the saturation of $C_{\zip}$. Then by pullback via $\zeta$, we clearly have an inclusion $\Ccal_{\zip}\subset \Ccal_X$. We conjectured:
\begin{conj1}[{\cite[Conjecture 2.1.6]{Goldring-Koskivirta-global-sections-compositio}}]
One has $\Ccal_X = \Ccal_{\zip}$.
\end{conj1}
For Hilbert--Blumenthal Shimura varieties, Conjecture 1 was verified by Goldring and the author in \cite{Goldring-Koskivirta-global-sections-compositio} and proved independently by Diamond--Kassaei (\cite[Corollary 8.2]{Diamond-Kassaei-cone-minimal}). In this case, an even stronger result is true: Both cones coincide with the subcone $\Ccal_{\pha}\subset \Ccal_{\zip}$ of partial Hasse invariants (the cone spanned by the weights $\ha_\tau$ defined in the beginning of the introduction). Imai and the author characterized precisely in \cite{Imai-Koskivirta-zip-schubert} the cases when the equality $\Ccal_{\pha}=\Ccal_{\zip}$ holds: those for which the parabolic $P$ is defined over $\FF_p$ and the action of the Frobenius on the roots of $L$ are given by $-w_{0,L}$, where $w_{0,L}$ is the longest element of the Weyl group $W_L$ of $L$. In particular, this condition does not hold for any group of $A_1$-type outside of the Hilbert--Blumenthal case. Imai and the author defined in \loccit another subcone $\Ccal_{\lwsf}\subset \Ccal_{\zip}$, called the lowest weight cone, related to the lowest weight vector of the induced representation $V(\mathbf{k})$. In this article, we show:
\begin{thmd}
Conjecture 1 holds for all Shimura varieties of $A_1$-type. More precisely, one has \[\Ccal_{\lwsf}=\Ccal_{\zip}=\Ccal_X=\Ccal_{E_n}\cap X_{+,L}^*(T).\]
\end{thmd}

To conclude this introduction, we give an overview of each section. In section 2, we review the theory of Shimura varieties, the Ekedahl--Oort stratification, the stack of $G$-zips. Section 3 is dedicated to the various cones that will play a role in this paper, namely the intersection-sum cone $\Ccal^{\cap,+}_S$, the cone of partial Hasse invariants $\Ccal_{\pha,S}$, the zip cone $\Ccal_{\zip}$, etc. In section 4, we introduce the notion of $p$-cones and study their properties (we define the terms $S$-adapted, homogeneous, admissible, positive). Finally, section 5 contains the computations necessary to prove Theorem C, which is the main technical result of the paper. Theorem B and D are derived as consequences of Theorem C.

\section*{Acknowledgements}
We would like to thank Wushi Goldring for useful discussions on several topics related to this article. The key idea of intersection-sum cone stems from joint work of Goldring and the author.

\section{Shimura varieties}

\subsection{Prelimenaries}
Let $(\mathbf{G},\mathbf{X})$ be a Shimura datum of Hodge-type (\cite{Deligne-Shimura-varieties}). In particular, $\mathbf{G}$ is a connected, reductive group over $\QQ$. For each neat compact open subgroup $K\subset \mathbf{G}(\AA_f)$, the Shimura variety $\Sh_K(\mathbf{G},\mathbf{X})$ is a quasi-projective scheme defined over a number field $\mathbf{E}$ (the reflex field of $(\mathbf{G},\mathbf{X})$). Let $p$ be a prime number and $v|p$ a place of $\mathbf{E}$. We say that $p$ is a prime of good reduction for $\Sh_K(\mathbf{G},\mathbf{X})$ if $K$ can be written as $K=K_p K^p$ where $K_p\subset \mathbf{G}(\QQ_p)$ is hyperspecial and $K^p\subset \mathbf{G}(\AA^p_f)$ is compact open. The condition on $K_p$ means that the group $\mathbf{G}_{\QQ_p}$ is unramified and that $K_p=\Gcal(\ZZ_p)$ for some reductive model $\Gcal$ of $\mathbf{G}_{\QQ_p}$ over $\ZZ_p$. In this case, Kisin (\cite{Kisin-Hodge-Type-Shimura}) and Vasiu (\cite{Vasiu-Preabelian-integral-canonical-models}) have shown that for each place $v|p$ in $\mathbf{E}$ there exists a smooth integral model $\Scal_K$ defined over $\Ocal_{\mathbf{E}_v}$ that is canonical (in the sense of Milne). The main object that we study in this paper is the special fiber $S_K\colonequals \Scal_K\otimes_{\Ocal_{\mathbf{E}_v}}\overline{\FF}_p$, which is a smooth, quasi-projective variety over $\overline{\FF}_p$. Madapusi-Pera has constructed smooth toroidal compactifications $\Scal_K^\tor$ of $\Scal_K$ in \cite{Madapusi-Hodge-tor}, for a choice of a sufficiently fine cone decomposition that we omit in the notation.

\subsection{\texorpdfstring{Ekedahl--Oort strata and stack of $G$-zips}{}} 
\label{EO-Gzip-sec}
Fix an algebraic closure $k$ of $\FF_p$. Let $G$ denote the connected reductive $\FF_p$-group $\Gcal\otimes_{\ZZ_p} \FF_p$. Denote by $\varphi \colon G\to G$ the Frobenius homomorphism. The Shimura datum $(\mathbf{G},\mathbf{X})$ induces a conjugacy class of cocharacters of $G_k$. We fix a represententive $\mu\colon \GG_{\mathrm{m},k}\to G_k$. From $\mu$, we obtain a pair of opposite parabolics $P_{\pm}(\mu)$, where  $P_+(\mu)(k)$ (resp. $P_-(\mu)(k)$) consists of the elements $g\in G(k)$ such that the map 
\begin{equation}\label{eq-Pmu}
\GG_{\mathrm{m},k} \to G_{k}; \  t\mapsto\mu(t)g\mu(t)^{-1} \quad (\textrm{resp. } t\mapsto\mu(t)^{-1}g\mu(t))   
\end{equation}
extends to a regular map $\AA_{k}^1\to G_{k}$. The centralizer of $\mu$ is a Levi subgroup $L(\mu)=P_+(\mu)\cap P_-(\mu)$. Define $P\colonequals P_-(\mu)$, $Q\colonequals (P_+(\mu))^{(q)}$, $L\colonequals L(\mu)$ and $M\colonequals L^{(p)}$. Let $\varphi\colon L\to M$ be the Frobenius homomorphism. Define the zip group $E$ by:
\begin{equation}\label{zipgroup}
E \colonequals \{(x,y)\in P\times Q \mid  \varphi(\theta^P_L(x))=\theta^Q_M(y)\}
\end{equation}
where $\theta^P_L\colon P\to L$ denotes the map sending $x\in P$ to its Levi component $\overline{x}\in L$ (and similarly for $\theta^Q_M$). Pink--Wedhorn--Ziegler introduced the stack of $G$-zips of type $\mu$, denoted by $\GZip^\mu$ in \cite[Definition 1.4]{Pink-Wedhorn-Ziegler-F-Zips-additional-structure}. It can be defined as the quotient stack
\begin{equation}
    \GZip^\mu=\left[E\backslash G_k \right].
\end{equation}
where $E$ acts on $G_k$ by $(x,y)\cdot g\colonequals xgy^{-1}$ for all $(x,y)\in E$ and all $g\in G$. One can also interpret it as a moduli stack of certain torsors. It is a smooth stack over $k$ whose underlying topological space is finite. Zhang has shown that there exists a smooth morphism
\begin{equation}\label{zip-map}
    \zeta\colon S_K\to \GZip^\mu.
\end{equation}
This map is also surjective on each connected component of $S_K$, as explained in \cite[\S 1.1.4]{Goldring-Koskivirta-GS-cone}. Goldring and the author showed that this map extends to a morphism $\zeta^{\tor}\colon S_K^{\tor}\to \GZip^\mu$. Furthermore, Andreatta proved in \cite{Andreatta-modp-period-maps} that the extension $\zeta^{\tor}$ is smooth.

The fibers of the map \eqref{zip-map} are called the Ekedahl--Oort strata (EO strata for short) of $S_K$. They are smooth, locally closed subsets of $S_K$. In the case of Shimura varieties of PEL-type, $S_K$ parametrizes abelian varieties endowed with a polarization, an action by a semisimple algebra and a level structure. In this case, two points $x,y\in S_K$ lie in the same EO stratum if and only if the abelian varieties $A_x$, $A_y$ corresponding to $x,y$ satisfy $A_x[p]\simeq A_y[p]$, where we require that the isomorphism is compatible with polarizations and the actions of the semisimple algebra.

\subsection{Parametrization of Ekedahl--Oort strata} \label{param-EO-sec}

Next, we give a parametrization of the points of $\GZip^\mu$. For convenience, we assume that there exists a Borel pair $(B,T)$ defined over $\FF_p$ such that $B\subset P$ and such that $\mu$ factors through $T$ (this can always be achieved after replacing $\mu$ by a conjugate cocharacter). Write $B^+$ for the opposite Borel of $B$, i.e. the unique Borel subgroup such that $B\cap B^+=T$. Let $\Phi^+\subset X^*(T)$ (resp. $\Phi^+_L$) be the set of positive $T$-roots in $G$ (resp. $L$), where we say that a $T$-root $\alpha$ is positive if the $\alpha$-root group $U_\alpha$ is contained in $B^+$. Write $\Delta$ (resp. $I\colonequals \Delta_L$) for the set of simple roots of $G$ (resp. $L$), and denote by $W$ (resp. $W_L$) be the Weyl group of $G$ (resp. $L$). For $\alpha \in \Phi$, let $s_\alpha$ be the corresponding root reflection. Write $\ell(w)$ for the length of an element $w\in W$ and let $\leq$ denote the Bruhat-Chevalley order on $W$. Write $w_0$ (resp. $w_{0,I}$) for the longest element of $W$ (resp. $W_L$). Let ${}^{I} W \subset W$ be the subset of elements $w\in W$ which are of minimal length in the right coset $W_I w$. It is a set of representatatives of the quotient $W_I \backslash W$. The maximal element of ${}^I W$ is $w_{0,I} w_0$. For $w\in W$, fix a representative $\dot{w}\in N_G(T)$ (where $N_G(T)$ is the normalization of $T$), such that $(w_1w_2)^\cdot = \dot{w}_1\dot{w}_2$ whenever $\ell(w_1 w_2)=\ell(w_1)+\ell(w_2)$ (this is possible by choosing a Chevalley system, \cite[ XXIII, \S6]{SGA3}). If no confusion occurs, we simply write $w$ instead of $\dot{w}$.

We write $X^*_+(T)$ for the set of dominant characters. We say that a character $\lambda\in X^*(T)$ is $L$-dominant if $\langle \lambda,\alpha^\vee\rangle \geq 0$ for all $\alpha\in I$. We denoted by $X^*_{+,I}(T)$ the set of $L$-dominant characters. Since $(B,T)$ is defined over $\FF_p$, all objects defined above are endowed with an action of the Galois group $\Gal(k/\FF_p)$. We write $\sigma\in \Gal(k/\FF_p)$ for the $p$-power Frobenius element. We set
\begin{equation}
    z\colonequals \sigma(w_{0,I})w_0.
\end{equation}
Since $\GZip^\mu=[E\backslash G_k]$, the points of the underlying topological space of $\GZip^\mu$ correspond bijectively to the $E$-orbits in $G_k$. Each such orbit is locally closed and smooth. For $w\in {}^I W$ set $G_w\colonequals E\cdot (z^{-1}w)$ (the $E$-orbit of $z^{-1}w$). Then, one has the following:

\begin{theorem}[{\cite[Theorem 7.5]{Pink-Wedhorn-Ziegler-zip-data}}] \label{thm-E-orb-param} \ 
\begin{assertionlist}
\item The map $w\mapsto G_w$ is a bijection from ${}^I W$ onto the set of $E$-orbits in $G_k$.
\item For $w\in {}^I W$, one has $\dim(G_w)= \ell(w)+\dim(P)$.
\item The Zariski closure of $G_w$ is 
\begin{equation}\label{equ-closure-rel}
\overline{G}_w=\bigsqcup_{w'\in {}^IW,\  w'\preccurlyeq w} G_{w'}
\end{equation}
where $\preccurlyeq$ is a certain partial order defined in \cite[\S 3,5]{Pink-Wedhorn-Ziegler-zip-data} that is finer than the Bruhat--Chevalley order. 
\end{assertionlist}
\end{theorem}

In particular, there is a unique open $E$-orbit
$U_\mu\subset G$ corresponding to the longest element $w_{0,I}w_0\in {}^I W$. For each $w\in {}^I W$, we put $\Xcal_w \colonequals [E\backslash G_w]$. It is a locally closed smooth substack of $\GZip^\mu=[E\backslash G_k]$. We obtain a stratification $\GZip^\mu = \bigsqcup_{w\in {}^I W} \Xcal_w$ and the closure relations between strata are given by \eqref{equ-closure-rel}. We also write $\Ucal_\mu\colonequals [E\backslash U_\mu]$ for the unique open stratum and call it the $\mu$-ordinary locus of $\GZip^\mu$, by analogy with Shimura varieties.

\subsection{The flag space}
\label{sec-flag}

Goldring and the author defined the stack of zip flags $\GF^{\mu}$ in \cite[\S2.1]{Goldring-Koskivirta-Strata-Hasse} based on previous work of Ekedahl--van der Geer (\cite{Ekedahl-Geer-EO}) in the case of Siegel-type Shimura varieties. It can be defined as a quotient stack 
\begin{equation}
    \GF^\mu \colonequals [E'\backslash G]
\end{equation}
where $E'=E\cap (B\times G)$. Since $E'\subset E$, there is a natural projection map $\pi \colon \GF^\mu \to \GZip^\mu$ whose fibers are isomorphic to $E/E'\simeq P/B$. Let $X$ be a $k$-scheme endowed with a morphism of stacks $\zeta \colon X\to \GZip^\mu$. Consider the fiber product
\begin{equation} \label{eq-def-flag-space}
\xymatrix@1@M=5pt{
\Flag(X) \ar[r]^-{\zeta_{\flag}} \ar[d]_-{\pi} & \GF^\mu \ar[d]^-{\pi} \\
X \ar[r]_-{\zeta} & \GZip^\mu
}
\end{equation} 
We call $\Flag(X)$ the flag space of $X$ (\cite[\S9.1]{Goldring-Koskivirta-Strata-Hasse}). In the case when $X$ is the special of a Hodge-type Shimura variety and $\zeta$ is Zhang's morphism \eqref{zip-map}, the space $\Flag(X)$ parametrizes pairs $(x,\Fcal_\bullet)$ where $x\in X$ and $\Fcal_\bullet$ is a full flag of $H^1_{\rm dR}(x)$ which refines the Hodge filtration. The map $\pi$ is simply the forgetful map $(x,\Fcal_\bullet)\mapsto x$. Next, we define the flag stratification on the flag space. Set $\Sbt\colonequals [B\backslash G /B]$ and call it the Schubert stack. By the Bruhat decomposition, the points of $\Sbt$ are parametrized by the elements $w\in W$ and are of the form \begin{equation}\label{schub-w}
    \Sbt_w\colonequals \left[ B\backslash BwB / B\right].
\end{equation}
By \cite[\S4.1]{Goldring-Koskivirta-zip-flags}, there is a natural smooth, surjective morphism of stacks
\begin{equation}
\label{psimap}
\Psi\colon \GF^\mu \to \Sbt\colonequals [B\backslash G /B].
\end{equation}
For all $w\in W$, we define $\Fcal_w$ to be the preimage $\Psi^{-1}(\Sbt_w)$ and call it the flag stratum of $w$. It is locally closed and smooth. Explicitly, it is the quotient stack
\begin{equation}
    \Flag_w=[E'\backslash BwBz^{-1}].
\end{equation}
Furthermore, the Zariski closure of $\Fcal_w$ is normal (since the Zariski closure of the Bruhat stratum $BwB$ is normal). Given a pair $(X,\zeta)$ as above, we obtain by pullback via $\zeta_{\flag}$ a locally closed stratification $(\Flag(X)_w)_{w\in W}$ on $\Flag(X)$, where
\begin{equation}
    \Flag(X)_w \colonequals \zeta^{-1}(\Fcal_w).
\end{equation}
In particular, this discussion applies to the pair $(S_K,\zeta)$ where $\zeta\colon S_K\to \GZip^\mu$ is Zhang's morphism \ref{zip-map}. We obtain a flag space $\Flag(S_K)$ attached to $S_K$. When $S_K$ is a Siegel-type Shimura variety, the flag space was first introduced and studied by Ekedahl--van der Geer in \cite{Ekedahl-Geer-EO}.

\subsection{\texorpdfstring{Shimura varieties of $A_1$-type}{}} \label{Shim-A1-type-sec}

We say that a reductive group $\mathbf{G}$ over a field $k$ is of $A_1$-type and rank $n$ if the adjoint group $\mathbf{G}^{\rm ad}_{\overline{k}}$ is a product of $n$ copies of $\PGL_{2,\overline{k}}$. We say that a Shimura variety is of $A_1$-type if it is attached to a reductive $\QQ$-group $\mathbf{G}$ of $A_1$-type. We will usually denote the special fiber of such a variety by the letter $X$, as the letter $S$ will have other use. There are mainly two examples of such Shimura varieties, that we describe below.

\paragraph{Hilbert--Blumenthal Shimura varieties} These varieties are sometimes also called Hilbert modular varieties. In this case, the group $\mathbf{G}$ is defined as follows: Let $\mathbf{F}/\QQ$ be a totally real extension of degree $n\colonequals[\mathbf{F}:\QQ]$. For any $\QQ$-algebra $R$, define
\begin{equation}
    \mathbf{G}(R)\colonequals \{ g\in \GL_2(\mathbf{F}\otimes_{\QQ} R), \ \det(g)\in R^\times \}.
\end{equation}
This defines a connected, reductive $\QQ$-group. It is the preimage of $\GG_{\mathrm{m},\QQ}$ under the natural determinant map $\det\colon \Res_{\mathbf{F}/\QQ}(\GL_{2,\mathbf{F}})\to \Res_{\mathbf{F}/\QQ}(\GG_{\mathrm{m},\mathbf{F}})$. If we write $\Sigma\colonequals \Hom_{\QQ}(\mathbf{F},\RR)$ for the set embeddings $\tau\colon \mathbf{F}\to \RR$, then $\mathbf{G}_{\overline{\QQ}}$ is naturally a subgroup of $\prod_{\tau\in \Sigma}\GL_{2,\overline{\QQ}}$, namely the group of tuples of matrices $(M_\tau)_{\tau\in \Sigma}$ satisfying that $\det(M_\tau)=\det(M_{\tau'})$ for any $\tau,\tau'\in \Sigma$. For a given compact open subgroup $K\subset \mathbf{G}(\AA_f)$, the Hilbert--Blumenthal Shimura variety $\Sh_K(\mathbf{G},\mathbf{X})$ is an $n$-dimensional quasi-projective variety that parametrizes principally polarized abelian varieties of rank $n$ with a compatible action of $\Ocal_{\mathbf{F}}$ and a $K$-level structure. Let $p$ be a prime number that is unramified in $\mathbf{F}$. Then $p$ is prime of good reduction for $\Sh_K(\mathbf{G},\mathbf{X})$ if $K=K_pK^p$ with $K_p\subset \mathbf{G}(\QQ_p)$ hyperspecial. Write $X$ for the special fiber of $\Sh_K(\mathbf{G},\mathbf{X})$ at $p$. The $\FF_p$-group $G$ depends on the ramification of $p$ in $\mathbf{F}$. Assume that $p\Ocal_{\mathbf{F}}$ decomposes as $\pfr_1\dots \pfr_r$ for (distinct) prime ideals $\pfr_i$ in $\Ocal_{\mathbf{F}}$ (recall that $p$ is unramified in $\mathbf{F}$). Writing $\kappa_i = \Ocal_{\mathbf{F}}/\pfr_i$ for the residual field of $\pfr_i$ and $f_i=[\kappa_i:\FF_p]$ for the residual degree of $\pfr_i$, we have $n=\sum_{i=1}^n f_i$ and $G$ is naturally a subgroup of the product
\begin{equation} \label{prod-Weyl}
    \prod_{i=1}^r \Res_{\kappa_i/\FF_p} \GL_{2,\kappa_i}.
\end{equation}
We make the following observation: Let $G_1,G_2$ be reductive groups over $\FF_p$ endowed with cocharacters $\mu_i\colon \GG_{\mathrm{m},k}\to G_{i,k}$ (for $i=1,2$), and assume $G=G_1\times G_2$. Define $\mu=(\mu_1,\mu_2)$. Then the stack $\GZip^\mu$ decomposes naturally as a direct product
\begin{equation}\label{Gzip-prod}
    \GZip^\mu = \GoneZip^{\mu_1}\times \GtwoZip^{\mu_2}.
\end{equation}
Thus, even though the Shimura variety $X$ does not decompose in this way as a direct product, all objects defined from the stack of $G$-zips will naturally decompose according to the $\FF_p$-factors of $G$. This makes it possible to reduce all group-theoretical arguments to the case $r=1$.

\paragraph{Unitary Shimura varieties}
In this case, we consider a CM-extension $\mathbf{E}/\mathbf{F}$ and set $n\colonequals [\mathbf{F}:\QQ]$. We let $\mathbf{V}$ be an $\mathbf{E}$-vector space of dimension $2$, endowed with a hermitian form $\psi\colon \mathbf{V}\times \mathbf{V}\to \mathbf{E}$. Let $\mathbf{G}$ be the group of unitary similitudes of $(\mathbf{V},\psi)$ (viewed as a reductive group over $\QQ$) with similitude factor in $\GG_{\mathrm{m}}$. Fix a CM-type $\{\tau_1,\dots,\tau_n\}\subset \Hom(\mathbf{E},\CC)$. For each embedding $1\leq i \leq n$, the signature of $\mathbf{V}\otimes_{\mathbf{F},\tau_i}\RR$ is $(r_i, 2-r_i)\in \{(0,2), (1,1), (2,0)\}$. The Shimura variety attached to $\mathbf{G}$ parametrizes principally polarized abelian varieties of dimension $2n$, endowed with a compatible action of $\Ocal_{\mathbf{F}}$ and a $K$-level structure (where $K\subset \mathbf{G}(\AA_f)$ is a compact open subgroup). One imposes the compatibility condition that the characteristic polynomial of $\alpha\in \Ocal_F$ acting on the Lie algebra of the abelian variety is given by
\begin{equation}
    \prod_{i=1}^{n}(X-\tau_i(\alpha))^{r_i}(X-\overline{\tau_i(\alpha)})^{2-r_i}
\end{equation}
where $z\mapsto \overline{z}$ denotes complex conjugation.

When $p$ is a prime number that is unramified in $\mathbf{F}$ and $K=K_pK^p$ with $K_p\subset \mathbf{G}(\QQ_p)$ hyperspecial, the Shimura variety $\Sh_K(\mathbf{G},\mathbf{X})$ has good reduction at $p$. The reductive $\ZZ_p$-model $\Gcal$ of $\mathbf{G}_{\QQ_p}$ is given by the choice of a $\Ocal_{\mathbf{E}}\otimes_{\ZZ} \ZZ_p$-lattice in $\mathbf{V}\otimes_{\QQ} \QQ_p$ which is self-dual for the form $\psi$. If we ignore the similitude factor for simplicity and look at the subgroup $\mathbf{G}'\colonequals \U(\mathbf{V},\psi)$ of unitary transformations, then the reduction modulo $p$ of $\mathbf{G}'$ (denoted by $G'$) is isomorphic over $\FF_p$ to a group of the form \eqref{prod-Weyl} appearing in the case of Hilbert--Blumenthal varieties. Again, the decomposition over $\FF_p$ of $G'$ is determined by the ramification of $p$ in $\mathbf{F}$. However, the case of unitary Shimura varieties differs from that of Hilbert--Blumenthal varieties in that the Hodge parabolic $P$ can be larger than the Borel subgroup $B$. Over $k$ we can canonically identify
\begin{equation}
    G'_{k} = \prod_{i=1}^n \GL_{2,k}.
\end{equation}
Then, $P$ (intersected with $G'_k$) decomposes as a product $P=\prod_{i=1}^n P_{i}$ where $P_i = \GL_{2,k}$ for $r_i\in \{0,2\}$ and $P_i$ is the Borel subgroup of lower-triangular matrices in $\GL_{2,k}$ when $r_i=1$. The set of $i\in \{1,\dots,n\}$ such that $r_i\in \{0,2\}$ will be called the parabolic type and will always be denoted by $R$. We will denote by $X$ or $X_R$ the special fiber of such unitary Shimura varieties. The dimension of $X_R$ is given by
\begin{equation}\label{dimXR}
    \dim(X_R)=n-|R|.
\end{equation}
To uniformize our results, we will always work with the reductive $\FF_p$-group \eqref{prod-Weyl} even though it is not exactly the group $G$ that appears in the setting of Hilbert--Blumenthal and unitary Shimura varieties. This change does not affect the stack of $G$-zips and is harmless in the formulation of our results.

\section{\texorpdfstring{Cones of mod $p$ automorphic forms}{}}

In this section, we recall the theory established in \cite{Imai-Koskivirta-zip-schubert} and \cite{Goldring-Koskivirta-GS-cone} regarding the various cones of weights that are naturally attached to a Shimura variety of Hodge-type.

\subsection{Automorphic vector bundles}\label{sec-autom-vect}
We keep the notation introduced in section \ref{param-EO-sec}. For any character $\lambda\in X^*(T)$, consider the induced representation
\begin{equation}
    V_I(\lambda)\colonequals \Ind_B^P(\lambda).
\end{equation}
Note that $V_I(\lambda)=0$ when $\lambda$ is not in the set $X^*_{+,I}(T)$ of $L$-dominant characters. Attached to $V_I(\lambda)$, there is an automorphic vector bundle $\Vcal_I(\lambda)$ defined on the Shimura variety $S_K$ (actually, this vector bundle can even be defined on the integral model $\Scal_K$). In our applications to $A_1$-type Shimura varieties, we will lighten the notation and write simply $\Vcal(\mathbf{k})$ for this vector bundle, where $\mathbf{k}\in \ZZ^n$. The rank of $\Vcal_I(\lambda)$ coincides with the dimension of $V_I(\lambda)$. Elements of the space of global sections
\begin{equation}
    H^0(S_K,\Vcal_I(\lambda))
\end{equation}
will be called mod $p$ automorphic forms of weight $\lambda$ and level $K$. It is natural to ask when this space vanishes. This is the question that has motivated Goldring and the author in a series of papers starting from \cite{Goldring-Koskivirta-global-sections-compositio}. To study this space, it is useful to consider the flag space $\Flag(S_K)$ defined in section \ref{sec-flag}. For each $\lambda\in X^*(T)$, there is a line bundle $\Vcal_{\flag}(\lambda)$ naturally attached to $\lambda$ on $\Flag(S_K)$ (in sections \ref{p-cone-sec} and \ref{coho-van-sec}, this line bundle will be denoted by $\Lcal(\mathbf{k})$ for $\mathbf{k}\in \ZZ^n$ to simplify the notation). If we denote by $\pi\colon \Flag(S_K)\to S_K$ the natural projection, we have
\begin{equation}\label{Vflag-push}
    \pi_{*}(\Vcal_{\flag}(\lambda))=\Vcal_I(\lambda).
\end{equation}
In particular, the space of mod $p$ automorphic forms of level $K$ and weight $\lambda$ identifies with the space of global sections of $\Vcal_{\flag}(\lambda)$ over $\Flag(S_K)$. Since the flag space admits a flag stratification $(\Flag(S_K)_w)_{w\in W}$, it is natural to introduce the following set:
\begin{equation}\label{CKw-def}
    C_{K,w}\colonequals \{\lambda\in X^*(T) \ | \ H^0(\overline{\Flag}(S_K)_w, \Vcal_{\flag}(\lambda)) \neq 0 \}
\end{equation}
where $\overline{\Flag}(S_K)_w$ is the Zariski closure of $\Flag(S_K)_w$ endowed with the reduced structure. We define $\Ccal_{K,w}$ as the saturation of $C_{K,w}$, i.e. the set of $\lambda\in X^*(T)$ that can be spanned over $\QQ_{\geq 0}$ by elements of $C_{K,w}$. We will always use the calligraphic letter $\Ccal$ to denote the saturation of a cone $C\subset X^*(T)$.

One of the motivations to introduce the various sets $\Ccal_{K,w}$ is that one can hope to determine them inductively starting at the elements $w\in W$ of length one and ending at the maximal element $w_0\in W$. For the maximal element, note that $\Flag(S_K)_{w_0}$ is open dense in $\Flag(S_K)$. Hence, since global sections of $\Vcal_{\flag}(\lambda)$ and $\Vcal_I(\lambda)$ coincide by \eqref{Vflag-push}, the set $C_{K,w_0}$ coincides with the following set
\begin{equation}\label{CK-def}
    C_{K}\colonequals \{\lambda\in X^*(T) \ | \ H^0(S_K,\Vcal_I(\lambda)) \neq 0 \}.
\end{equation}
Again, we write $\Ccal_K$ for the saturation of $C_K$. Since $\Vcal_I(\lambda)=0$ when $\lambda$ is not $L$-dominant, we have $\Ccal_K\subset X^*_{+,I}(T)$. However, the cones $\Ccal_{K,w}$ for $w\neq w_0$ are in general not contained in $X^*_{+,I}(T)$.

\subsection{Partial Hasse invariant cones for strata}\label{part-Hasse-sec}

In this section, we explain that to each stratum $\Flag(S_K)_w$, one can naturally attach a subset $\Ccal_{\pha,w}\subset \Ccal_{K,w}$ called the cone of partial Hasse invariants of $w$ (by "cone" we mean an additive monoid with zero). Recall that there is a morphism of stacks $\Psi\colon \GF^\mu\to \Sbt$ whose fibers define the flag stratification $(\Fcal_w)_{w\in W}$. Intuitively, the set $\Ccal_{\pha,w}$ is the set of characters $\lambda\in X^*(T)$ such that $\Vcal_{\flag}(\lambda)$ admits a nonzero section over $\overline{\Fcal}_w$ which arises by pullback from a nonzero section over $\overline{\Sbt}_w$. A much more explicit definition is the following: For each $w\in W$, write $E_w$ for the set of positive roots $\alpha\in \Phi^+$ such that 
\begin{equation}
    ws_\alpha < w \quad \textrm{and} \ \ell(ws_\alpha)=\ell(w)-1.
\end{equation}
Intuitively, the elements $ws_\alpha$ are the lower neighbors of $w$ in the Weyl group $W$ with respect to the Bruha--Chevalley order. Write $X^*_{+,w}(T)$ for the set of characters $\lambda\in X^*(T)$ such that $\langle \lambda, \alpha^\vee \rangle \geq 0$ for all $\alpha \in E_w$.

\begin{definition}[{\cite[\S 2.3]{Goldring-Koskivirta-GS-cone}}] \label{def-Cphaw}
We define the cone $C_{\pha,w}$ of $w$ as the direct image of the set $X^*_{+,w}(T)$ under the map
\begin{equation}
    h_w\colon X^*(T)\to X^*(T), \quad \lambda \mapsto -w\lambda + p w_{0,I}w_0 \sigma^{-1}(\lambda).
\end{equation}
Moreover, we write $\Ccal_{\pha,w}$ for the saturation of $C_{\pha,w}$ and call it the cone of partial Hasse invariants of $w$.
\end{definition}
The intuitive interpretation of $\Ccal_{\pha,w}$ explained in the beginning of this paragraph implies immediately that $\Ccal_{\pha,w}$ is contained in $\Ccal_{K,w}$. When $w=w_0$, we simply write $\Ccal_{\pha}$ for the cone $\Ccal_{\pha,w_0}$.

\begin{rmk} \label{rmk-triv-S}
If $\chi\in X^*(T)$ satisfies $\langle \chi,\alpha^\vee\rangle=0$ for all $\alpha\in E_w$, then if we write $\lambda=h_w(\chi)$, the line bundle $\Vcal_{\flag}(\lambda)$ admits a nowhere-vanishing section (see \cite[\S 2.3]{Goldring-Koskivirta-GS-cone}), and thus is trivial on $\overline{\Flag}(S_K)_w$.    
\end{rmk}

\subsection{Intersection-sum cones}

Next, we attach to each element $w\in W$ another natural cone, called the intersection-sum cone, denoted by $\Ccal_{w}^{\cap,+}$. It was first introduced by Goldring and the author in \cite{Goldring-Koskivirta-divisibility, Goldring-Koskivirta-GS-cone}. The set $\Ccal_{w}^{\cap,+}$ is defined inductively on the length of $w$ as follows: When $\ell(w)=1$, we simply set $\Ccal_{w}^{\cap,+}=\Ccal_{\pha,w}$. Then, for any $w\in W$ of length $\geq 2$, define
\begin{equation}
    \Ccal_{w}^{\cap,+} = \Ccal_{\pha,w} \ +_{\rm sat} \ \bigcap_{\alpha \in E_w} \Ccal_{w s_\alpha}^{\cap,+} 
\end{equation}
Here, for two cones $\Ccal_1, \Ccal_2\subset X^*(T)$, we denote by $\Ccal_1 +_{\rm sat} \Ccal_2$ the saturation of the sum $\Ccal_1 + \Ccal_2$. The definition of $\Ccal_w^{\cap,+}$ combines information about $w$ and of its lower neighbors. To explain the relevance of this definition, we make the following assumption for simplicity:
\begin{assumption}\label{assume-hasse}
For all $w\in W$, the elements $\alpha^\vee$ are linearly independent over $\QQ$ in $X_*(T)\otimes_{\ZZ} \QQ$.
\end{assumption}
This assumption is rarely satisfied. However, for general groups, one can slightly alter the definition of $\Ccal_{w}^{\cap,+}$ to remove the need of Assumption \ref{assume-hasse} (this is work in progress with Goldring \cite{Goldring-Koskivirta-orthogonal}). However, note the above assumption is satisfied in the case when $G$ is an $A_1$-type reductive group. The main property of intersection-sum cones is the following:

\begin{theorem}[{\cite[Theorem 2.3.8]{Goldring-Koskivirta-divisibility}}]\label{thm-sep-syst}
Under Assumption \ref{assume-hasse}, one has for all $w\in W$ an inclusion
\begin{equation}
    \Ccal_{K,w} \subset \Ccal_{w}^{\cap,+}.
\end{equation}
\end{theorem}
In other words, this theorem says that the space $H^0(\overline{\Flag}(S_K)_w,\Vcal_{\flag}(\lambda))$ is zero whenever $\lambda$ lies in the complement of $\Ccal_{w}^{\cap,+}$. Let us consider the case $w=w_0$. By Theorem \ref{thm-sep-syst}, we have an inclusion $\Ccal_{K,w_0} \subset \Ccal_{w_0}^{\cap,+}$. Also, recall that $\Ccal_{K,w_0}$ coincides with $\Ccal_K$ defined in section \ref{sec-autom-vect}. Furthermore, we know that $\Ccal_K$ is contained in the set of $L$-dominant characters. Thus, we deduce the more precise inclusion:
\begin{equation}\label{precise-inclu-w0}
    \Ccal_K\subset \Ccal_{w_0}^{\cap,+} \cap X^*_{+,I}(T).
\end{equation}

\subsection{The zip cone}
\label{sec-zipcone}

So far we have attached several cones to each flag stratum $\Flag(S_K)_w$ for $w\in W$, satisfying certain inclusions. Although we will not use it in full, we mention here one more natural construction that one could consider in this setting. In the definition of the cones $\Ccal_{K,w}$ and $\Ccal_K$ (see \eqref{CKw-def}, \eqref{CK-def}), one could replace the flag stratum $\Flag(S_K)_w$ and the scheme $S_K$ by their "group-theoretical counterparts", namely the stratum $\Fcal_w$ and the stack $\GZip^\mu$ respectively. By modifying in this way the cone $\Ccal_{K,w}$, we obtain another cone, that we denote by $\Ccal_{\flag,w}$. Similarly, the modified version of $\Ccal_K$ is denoted by $\Ccal_{\zipsf}$ and is called the zip cone. Explicitly, $\Ccal_{\flag,w}$ and $\Ccal_{\zipsf}$ are respectively the saturations of the following two sets:
\begin{align*}\label{Cflagw-def}
    C_{\flag,w}&\colonequals \{\lambda\in X^*(T) \ | \ H^0(\overline{\Fcal}_w, \Vcal_{\flag}(\lambda)) \neq 0 \} \\ 
    C_{\zip}&\colonequals \{\lambda\in X^*(T) \ | \ H^0(\GZip^\mu,\Vcal_I(\lambda)) \neq 0 \}.
\end{align*}
If we denote by $\rho_\lambda\colon P\to \GL_k(V_I(\lambda))$ the homomorphism given by the representation $V_I(\lambda)$, the space $H^0(\GZip^\mu,\Vcal_I(\lambda))$ is the $k$-vector space of regular maps $f\colon G_k\to V_I(\lambda)$ satisfying the condition
\begin{equation}\label{zip-form-eq}
f(axb^{-1})=\rho_\lambda(a)f(x)
\end{equation}
for all $(a,b)\in E$ and all $x\in G_k$. Under Assumption \ref{assume-hasse}, we have inclusions
\begin{equation}\label{inclu-w}
    \Ccal_{\pha,w} \ \subset \ \Ccal_{\flag,w} \ \subset \ \Ccal_{K,w} \ \subset \Ccal_{w}^{\cap,+}.
\end{equation}
Similarly to \eqref{Vflag-push}, the equation $\pi_*(\Vcal_{\flag}(\lambda))=\Vcal_I(\lambda)$ holds also for the natural projection map $\pi\colon \GF^\mu \to \GZip^\mu$. Therefore, similarly to the equality $\Ccal_{K,w_0}=\Ccal_K$, we have $\Ccal_{\flag,w_0}=\Ccal_{\zipsf}$. Hence, for $w=w_0$, we have inclusions
\begin{equation}\label{all-inclu-w0}
    \Ccal_{\pha} \ \subset \ \Ccal_{\zipsf} \ \subset \ \Ccal_{K} \ \subset \Ccal_{w_0}^{\cap,+}\cap X^*_{+,I}(T).
\end{equation}
Goldring and the author proposed in \cite{Goldring-Koskivirta-global-sections-compositio}[Conjecture 2.1.6] the following conjecture, that we refer to simply as the Cone Conjecture:
\begin{conjecture}\label{conj-cones}
For any Hodge-type Shimura variety, one has $\Ccal_{K} = \Ccal_{\zipsf}$. 
\end{conjecture}
In other words, this conjecture states that the vanishing of the degree zero cohomology of automorphic vector bundles is encoded by the stack of $G$-zips. It was checked in \cite{Koskivirta-automforms-GZip} that the cone $\Ccal_{K}$ is independent of the level $K$ (however the cone $C_K$ does depend on $K$). Thus it is not unreasonable to expect that such an equality holds. The above conjecture was proved for several Shimura varieties in \cite{Goldring-Koskivirta-global-sections-compositio, Goldring-Koskivirta-divisibility}: Hilbert--Blumenthal varieties, Siegel modular varieties of rank $2$ and $3$, unitary Shimura varieties of rank $\leq 4$ (with the exception of the case of signature $(2,2)$ at an inert prime). The case of Hilbert--Blumenthal varieties was also proved by Diamond--Kassaei (\cite[Corollary 8.2]{Diamond-Kassaei-cone-minimal}).

One of the results of this paper is the following:

\begin{theorem}
Conjecture \ref{conj-cones} holds for all Hodge-type Shimura varieties of $A_1$-type.
\end{theorem}

We give the proof of this theorem in section \ref{sec-zip-proof}. It requires the use of a family of mod $p$ automorphic forms attached to the lowest weight vectors of the representations $V_I(\lambda)$ (see section \ref{low-cone-sec}). We explain below a strategy to prove Conjecture \ref{conj-cones} using the intersection-sum cones (under Assumption \ref{assume-hasse}). In view of the inclusions \eqref{all-inclu-w0}, it clearly suffices to show that $\Ccal_{\zipsf}$ coincides with $\Ccal_{w_0}^{\cap,+}\cap X^*_{+,I}(T)$. We will show that this is true when $G$ is a group of $A_1$-type. One advantage of this method is that it behaves well under $\FF_p$-products. Specifically, assume that $G=G_1\times G_2$ where $G_1$, $G_2$ are reductive $\FF_p$-groups and that $\mu = (\mu_1,\mu_2)$ for two cocharacters $\mu_{i}\colon \GG_{\mathrm{m},k}\to G_{i,k}$ (for $i=1,2$). As we explained in \eqref{Gzip-prod}, the stack $\GZip^\mu$ splits into a direct product of $\GoneZip^{\mu_1}$ and $\GtwoZip^{\mu_2}$. If we write $T=T_1\times T_2$ where $T_i$ is a maximal torus in $G_i$ (for $i=1,2$), then all the cones $C_{\pha,w}$, $\Ccal_w^{\cap,+}$, $\Ccal_{\flag,w}$, $C_{\zipsf}$ decompose as $C_1\times C_2$ where $C_i\subset X^*(T_i)$ is the corresponding subcone. Note that since $S_K$ does not split in this way, we can a priori say nothing about the cone $\Ccal_K$. We deduce that if the equality $\Ccal_{\zipsf} = \Ccal_{w_0}^{\cap,+}\cap X^*_{+,I}(T)$ holds for all the $\FF_p$-factors of $G$, then it also holds for $G$, and thus Conjecture \ref{conj-cones} holds.

Finally, we end this section with a word on the cone $\Ccal_{\pha}$ of partial Hasse invariants (for the maximal stratum). For Hilbert--Blumenthal varieties, Siegel threefolds and Picard surfaces, Goldring and the author proved in \cite{Goldring-Koskivirta-global-sections-compositio} a more precise statement than Conjecture \ref{conj-cones}, namely that $\Ccal_{\pha}=\Ccal_{\zip}=\Ccal_{K}$. However, the inclusion $\Ccal_{\pha}\subset \Ccal_{\zip}$ is in general strict. In \cite{Imai-Koskivirta-zip-schubert}, Imai and the author characterized exactly when this inclusion is an equality:
\begin{theorem} \label{thm-Hasse-type}
The following are equivalent:
\begin{equivlist}
\item One has $\Ccal_{\pha} = \Ccal_{\zipsf}$.
\item The parabolic $P$ is defined over $\FF_p$ and $\sigma$ acts on $I$ by $-w_{0,I}$.
\end{equivlist}
\end{theorem}
For groups of $A_1$-type studied in this paper, the parabolic $P$ will in general not be defined over $\FF_p$. In those cases, the above theorem shows that $\Ccal_K$ is not spanned by partial Hasse invariants.

\subsection{Lowest weight cone}\label{low-cone-sec}

As mentioned above, partial Hasse invariants do not suffice to describe the cone $\Ccal_K$ of weights of mod $p$ automorphic forms in general. We introduce another family of natural mod $p$ automorphic forms that will be necessary. We explain their construction on the stack of $G$-zips. First, recall that the unique open stratum in $\GZip^\mu$ is $\Ucal_\mu\colonequals [E\backslash U_\mu]$ (see section \ref{param-EO-sec}). The Zariski open subset $U_\mu\subset G$ coincides with the $E$-orbit of $1\in G$. Moreover, the stabilizer of $1$ in $E$ is a finite (in general non-smooth) group scheme $L_\varphi$, which embeds naturally into the Levi subgroup $L\subset P$ via the first projection $\pr_1\colon E\to P$. Denote by $P_0$ the largest algebraic subgroup contained in $P$ and defined over $\FF_p$, in other words $P_0=\bigcap_{i\in \ZZ} \sigma^i(P)$. Since $B$ is defined over $\FF_p$, $P_0$ contains $B$ and is a parabolic subgroup. Write $L_0$ for the unique Levi subgroup of $P_0$ containing $T$ and $I_0\subset I$ for the simple roots of $L_0$. The underlying reduced group scheme of $L_\varphi$ is the etale group scheme $L_0(\FF_p)$. Since $U_\mu\simeq E/L_\varphi$, we obtain $\Ucal_\mu = [E\backslash E/L_\varphi] \simeq [1/ L_\varphi]$. It follows that we have an identification
\begin{equation}\label{ident-Umu}
    H^0(\Ucal_\mu,\Vcal_I(\lambda)) \simeq V_I(\lambda)^{L_\varphi}.
\end{equation}
Our approach is to first construct a section on the open substack $\Ucal_\mu$ an then determine explicit conditions when this section extends to $\GZip^\mu$. For any $f\in V_I(\lambda)$, we can define the $L_\varphi$-norm $\Norm_{L_\varphi}(f)$ of $f$ by taking the product of the translates of $f$ by the elements of $L_{\varphi}$ (see \cite[\S 3.5]{Imai-Koskivirta-zip-schubert}). The $L_\varphi$-norm of $f$ lies in $V_I(N\lambda)$ where $N$ is the order of $L_\varphi$. Thus, via the identification \eqref{ident-Umu}, the norm $\Norm_{L_\varphi}(f)$ identifies with a section over $\Ucal_\mu$, and hence by pullback via $\zeta$ to a "rational" automorphic form of weight $N\lambda$ defined over the $\mu$-ordinary locus of $S_K$.

In the case when $f=f_{\lambda,\high}$ is the highest weight vector of $V_I(\lambda)$, we can determine exactly when $\Norm_{L_\varphi}(f_{\lambda,\high})$ extends to a global section. For each root $\alpha\in \Phi$, let $r_\alpha\geq 1$ be an integer such that $\sigma^{r_\alpha}(\alpha)=\alpha$.
\begin{proposition}[{\cite[Proposition 3.5.1]{Imai-Koskivirta-zip-schubert}}]\label{prop-Norm}
The section $\Norm_{L_\varphi}(f_{\high,\lambda})$ extends to a global section if and only if for all $\alpha \in \Delta^P$, the following holds:
\begin{equation}\label{formula-norm}
\sum_{w\in W_{L_0}(\FF_q)} \sum_{i=0}^{r_\alpha-1} p^{i+\ell(w)} \ \langle w\lambda, \sigma^i(\alpha^\vee) \rangle\leq 0.
\end{equation}
\end{proposition}
For the lowest weight vector $f_{\low,\lambda}$ of $V_I(\lambda)$, this problem is surprisingly much more difficult. We need to make a technical assumption on the group $G$ (Condition 5.1.1 of \cite{Imai-Koskivirta-zip-schubert}). We do not recall this condition here, but we simply note that it is satisfied for groups of the form $G=\Res_{\FF_{p^n}/\FF_p}(G_{0,\FF_{p^n}})$ where $G_0$ is a split reductive over $\FF_p$. In particular, it holds for groups of $A_1$-type considered in this paper.
\begin{theorem}\label{thm-norm-low} 
Define $\lambda_0\colonequals w_{0,I_0}w_{0,I}\lambda$. Suppose that for all $\alpha \in \Delta^{P_0}$, one has
\begin{equation}\label{formula-norm-low}
\sum_{w\in W_{L_0}(\FF_p)} \sum_{i=0}^{r_\alpha-1} p^{i+\ell(w)} \ \langle w \lambda_0, \sigma^i(\alpha^\vee) \rangle\leq 0.
\end{equation}
Then $\Norm_{L_\varphi}(f_{\low,\lambda})$ extends to a global section.
\end{theorem}
Define $\Ccal_{\hw}$ (resp. $\Ccal_{\lw}$) as the set of $\lambda\in X^*_{+,I}(T)$ satisfying the inequalities \eqref{formula-norm} (resp. \eqref{formula-norm-low}). By the above results, both $\Ccal_{\hw}$ and $\Ccal_{\lw}$ are contained in $\Ccal_{\zip}$. In general, there is no inclusion relation between these cones and the cone $\Ccal_{\pha}$ defined earlier. When $P$ is defined over $\FF_p$, one sees easily that $\Norm_{L_\varphi}(f_{\low,\lambda})=\Norm_{L_\varphi}(f_{\high,\lambda})$ and $\Ccal_{\lw}=\Ccal_{\hw}$. However, in general the cone $\Ccal_{\lw}$ seems to be the more relevant of the two.

\subsection{Torsion line bundles} \label{sec-torsion}

Let $\lambda\in X^*(L)$ be a character of $L$, i.e. an algebraic group homomorphism $\lambda\colon L\to \GG_{\mathrm{m},k}$. In this case the representation $V_I(\lambda)$ is one-dimensional, and coincides with $\lambda$ itself. Thus, for $\lambda\in X^*(L)$ the vector bundle $\Vcal_I(\lambda)$ is a line bundle on $\GZip^\mu$ and $S_K$. Note that for any $\lambda,\lambda'\in X^*(L)$ we have $\Vcal_I(\lambda+\lambda')=\Vcal_I(\lambda)\otimes \Vcal_I(\lambda')$. By \eqref{zip-form-eq}, a global section of $\Vcal_I(\lambda)$ (for $\lambda\in X^*(L)$) is simply a map $f\colon G_k\to \AA^1$ satisfying
\begin{equation}\label{line-bdl-eq}
    f(axb^{-1})=\lambda(a)f(x)
\end{equation}
for all $(a,b)\in E$ and all $x\in G_k$. Let us examine the special case when $\lambda\in X^*(G)$.

\begin{lemma}
For all $\lambda\in X^*(G)$, the line bundle $\Vcal_{I}(\lambda)$ is torsion, in the sense that $\Vcal_{I}(\lambda)^{\otimes m}=\Vcal_{I}(m\lambda)=\Ocal_{S_K}$ for some integer $m\geq 1$.
\end{lemma}

\begin{proof}
    
Take $f$ to be the character $\lambda\colon G_k\to \GG_{\mathrm{m}}\subset \AA^1$. For any $x\in G_k$, we have $\lambda(axb^{-1})=\lambda(a)\lambda(b)^{-1}\lambda(x)$. By definition of the group $E$, for $(a,b)\in E$ the Levi component of $b$ is the Frobenius homomorphism applied to the Levi component of $a$. Thus, we have $\lambda(b)=\lambda(\varphi(a))=(\sigma^{-1} \lambda)(a)^p$. This implies that $f$ is a section of the line bundle $\Vcal_I(\lambda-p\sigma^{-1}(\lambda))$. Since $f$ is non-vanishing, this implies that $\Vcal_I(\lambda-p\sigma^{-1}(\lambda))$ is trivial. The map $X^*(G)\to X^*(G)$, $\lambda\mapsto \lambda-p\sigma^{-1}(\lambda)$ is injective (because its mod $p$ reduction is an isomorphism), hence there is an integer $m\geq 1$ such that $m X^*(G)$ is contained in the image of this map. This shows that for any $\lambda\in X^*(G)$, the line bundle $\Vcal_I(\lambda)^{\otimes m}$ is trivial.
\end{proof}

Furthermore, one can see easily see that $X^*(G)$ is contained in all the (saturations of the) cones defined in this section. It will be convenient to ignore these torsion line bundles and quotient out all objects by $X^*(G)$. Hence, we will view all the cones as subsets of $X^*(T)/X^*(G)$.

\subsection{Asymptotic behaviour}\label{sec-asymp}
This section discusses heuristically the behaviour of the various cones when the prime number $p$ goes to infinity. Let $\lambda\in X^*(T)$ be a character and assume that $\Vcal_{\flag}(\lambda)$ is ample on the flag space $\Flag(S_K)$. Then by ampleness, $\Vcal_{\flag}(\lambda)$ has nonzero sections on any closed subscheme of $\Flag(S_K)$. In particular, $\lambda$ lies in the cones $\Ccal_{K,w}$ for all $w\in W$ and in $\Ccal_K$. It is difficult to determine which line bundles $\Vcal_{\flag}(\lambda)$ are ample (see \cite{Alexandre-Vanishing, Brunebarbe-Goldring-Koskivirta-Stroh-ampleness} for partial results). However, it is known to experts that in characteristic zero, the line bundle $\Vcal_{\flag}(\lambda)$ is ample if and only if $\lambda$ lies in the following set
\begin{equation}
\Ccal_{\GSsf}^{\circ}=\left\{ \lambda\in X^{*}(T) \ \relmiddle| \ 
\parbox{6cm}{
$\langle \lambda, \alpha^\vee \rangle > 0 \ \textrm{ for }\alpha\in I, \\
\langle \lambda, \alpha^\vee \rangle < 0 \ \textrm{ for }\alpha\in \Phi^+ \setminus \Phi_{L}^{+}$}
\right\}.
\end{equation}
If a line bundle is ample in characteristic zero, then it is also ample modulo $p$ for large $p$, because ampleness is an open condition on the basis. Thus for $\lambda\in \Ccal_{\GSsf}^{\circ}$, the line bundle $\Vcal_{\flag}(\lambda)$ is ample for large $p$ on $\Flag(S_K)$. Thus, when $p$ goes to infinity we expect $\Ccal_{\GSsf}^{\circ}$ to be contained "at the limit" in the cones $\Ccal_{K,w}$ for all $w\in W$.

For the maximal element $w_0$, we can say something stronger. For $\lambda\in \Ccal_{\GSsf}^{\circ}$, choose any nonzero section $f$ of $\Vcal_I(\lambda)$. By a reduction mod $p$ argument (\cite[Proposition 1.8.3]{Koskivirta-automforms-GZip}), one can show that there exists a nonzero global section of $\Vcal_I(\lambda)$ in characteristic $p$ as well. This implies that the set $\Ccal_{\GSsf}^{\circ}$ is contained in $\Ccal_K$ for all $p$ (not only for large $p$). Since Conjecture \ref{conj-cones} predicts that $\Ccal_K=\Ccal_{\zip}$, this set should also be contained in $\Ccal_{\zip}$. In \cite{Imai-Koskivirta-zip-schubert}, Imai and the author showed a more precise result: Denote by $\Ccal_{\GS}$ the closure of $\Ccal_{\GS}^{\circ}$ (i.e. replace the strict inequalities in $\Ccal_{\GS}^{\circ}$ by inclusive inequalities). Then we proved the following for a general connected reductive group $G$ over $\FF_p$:
\begin{theorem}[{\cite[Theorem 6.4.3]{Imai-Koskivirta-zip-schubert}}]
One has $\Ccal_{\GS}\subset \Ccal_{\zip}$.
\end{theorem}

\section{\texorpdfstring{$p$-cones}{}}\label{p-cone-sec}

In this section, we introduce the theory of $p$-cones. This theory is for the most part independent of group-theoretical considerations, but it is tailored to the case of a Weyl restriction $\Res_{\FF_p^n/\FF_n}(\GL_{2,\FF_{p^n}})$ for a fixed integer $n\geq 1$. As we explained in section \ref{Shim-A1-type-sec} and \ref{sec-zipcone}, we can reduce the proofs of our results to this fundamental case.

\subsection{Preliminaries}
\label{thegroup}
Let $n\geq 1$ be an integer and let $G$ be the group
\begin{equation}
    G\colonequals \Res_{\FF_p^n/\FF_n}(\GL_{2,\FF_{p^n}}).
\end{equation}
Over the algebraic closure $k$, the group $G_k$ decomposes as a direct product of $n$ copies of $\GL_{2,k}$. We order the factors so that the action of the Frobenius $\sigma\in \Gal(k/\FF_p)$ is given on $G(k)$ by
\begin{equation}\label{gal-inert}
{}^\sigma (x_1, \dots ,x_n) \colonequals (\sigma(x_{2}), \dots ,\sigma(x_{n}),\sigma(x_1))
\end{equation}
where $\sigma(x)$ is the usual Frobenius action on $x\in \GL_2(k)$. Let $T$ (resp. $B$) denote the maximal torus (resp. Borel subgroup) consisting of $n$-tuples of diagonal (resp. lower-triangular) matrices. It is clear that the Borel pair $(B,T)$ is defined over $\FF_p$. The characters of $T$ form a $\ZZ$-module of rank $2n$. For a tuple $\mathbf{k}=(k_1,\dots,k_n)$, identify $\mathbf{k}$ with the character
\begin{equation}
  \chi_{\mathbf{k}} \colon T\to \GG_{\mathrm{m}}, \quad \left( \left(\begin{matrix}
        x_1 & \\ & y_1
    \end{matrix} \right), \dots , \left(\begin{matrix}
        x_n & \\ & y_n
    \end{matrix} \right)  \right) \mapsto \prod_{i=1}^n x_i^{k_i}.
\end{equation}
Set $\Lambda$ for the $\ZZ$-submodule of $X^*(T)$ consisting of characters of this form. Then we have a decomposition
\begin{equation}
    X^*(T) = X^*(G) \oplus \Lambda.
\end{equation}
As explained in section \ref{sec-torsion}, we ignore the torsion line bundles corresponding to characters in $X^*(G)$. Note that if $\Ccal\subset X^*(T)$ is a subcone containing $X^*(G)$, then we can write $\Ccal=X^*(G)\oplus (\Ccal\cap \Lambda)$. Thus, for simplicity we may intersect all subcones of $X^*(T)$ defined in the previous section with the submodule $\Lambda$. We identify $\Lambda$ with $\ZZ^n$ via the map $\mathbf{k}\mapsto \chi_{\mathbf{k}}$. For these reasons, we will define in the next section the notion of $p$-cone as a certain subcone of $\ZZ^n$.

We identify the Weyl group $W\colonequals W(G,T)$ with $W=\{\pm 1\}^n$. The maximal element $w_0\in W$ corresponds to $(-1,\dots,-1)$. An element $\varepsilon = (\varepsilon_1,\dots,\varepsilon_n)\in W$ acts on $\Lambda=\ZZ^n$ by 
\begin{equation}
    \varepsilon \cdot \mathbf{k} = (\varepsilon_1 k_1, \dots , \varepsilon_n k_n)
\end{equation}
for all $\mathbf{k}=(k_1, \dots , k_n)\in \ZZ^n$. The Frobenius element $\sigma\in \Gal(k/\FF_p)$ acts on $\Lambda=\ZZ^n$ by the rule
\begin{equation}
    \sigma \cdot \mathbf{k} = (k_n, k_1, \dots ,k_{n-1}).
\end{equation}
Write $e_1,\dots,e_n$ for the standard basis of $\ZZ^n$. The above formula shows that $\sigma e_i = e_{i+1}$ (where $e_{n+1}=e_1$). We put
\begin{equation}
    E_n\colonequals \{1,\dots,n\}.
\end{equation}
For any subset $S\subset E_n$ and any $i\in E_n$, set
\begin{equation}
    \delta_S^{(i)}\colonequals \begin{cases}
        -1 & \textrm{if }i\in S \\
        1 & \textrm{if }i\notin S.
    \end{cases}
\end{equation}
Define an element $\delta_S\in \{\pm 1\}^n$ by $\delta_S=(\delta_S^{(1)},\dots,\delta_S^{(n)})$. Write $\Pcal(E_n)$ for the powerset of $E_n$. The map $S\mapsto \delta_S$ is a bijection
\begin{equation}\label{delta-bij}
   \delta \colon \Pcal(E_n) \to W.
\end{equation}
The inverse of $\delta$ is the map that takes $\varepsilon\in \{\pm 1\}^n$ to the subset $\{i\in E_n \ | \ \varepsilon_i=-1\}$. 

We fix a subset $R\subset E_n$ and call it a parabolic type. The choice of $R$ gives rise to a parabolic subgroup $P=P_R$ of $G_k$ defined as $P=P_1\times \dots \times P_n$ where $P_i$ is the lower-triangular Borel of $\GL_{2,k}$ if $i\notin R$ and $P_i=\GL_{2,k}$ for $i\in R$. In particular, $P_{\emptyset} = B$ and $P_{E_n}=G$. It is the parabolic subgroup attached to the cocharacter $\mu_R\colon \GG_{\mathrm{m},k}\to G_k$ which is trivial on the factors $i\in R$ and is given by
\begin{equation}
    z\mapsto \left( \begin{matrix}
        z & \\ & 1
    \end{matrix}\right)
\end{equation}
on the factors $i\notin R$. The pair $(G,\mu)$ gives rise to a stack of $G$-zips as explained in section \ref{EO-Gzip-sec}. The element $w_{0,I}\in W_I$ corresponds to $\delta_R\in \{\pm 1\}^n$.

Independently of the theory of Shimura varieties, all the group-theoretical objects $\Ccal_{\pha,w}$, $\Ccal_{\flag,w}$, $\Ccal_{\zip}$, $\Ccal_w^{\cap,+}$ are defined as in the previous section and they are subcones of $X^*(T)$ containing $X^*(G)$. For applications to Shimura varieties, we need only remember that the group $G$ was slightly modified. In this section, we will replace all indices $w$ by the subset $S\subset E_n$ that $w$ corresponds to via the bijection \eqref{delta-bij}. Therefore, we write $\Ccal_{\pha,S}$, $\Ccal_{\flag,S}$, $\Ccal_S^{\cap,+}$ for the corresponding cones. To reiterate, we implicitly consider these subcones of $X^*(T)$ as subsets of $\ZZ^n$ by intersecting them with $\Lambda$ and identifying $\Lambda=\ZZ^n$. By abuse of language, a subset $S\subset E_n$ will sometimes be referred to as a stratum.

Even though this section is completely group-theoretical, we will sometimes mention the connection to Shimura varieties. Since the letter $S$ will always denote a subset of $E_n$, we will denote by $X$ the special fiber of a Hodge-type Shimura variety of $A_1$-type and rank $n$. To make the choice of the parabolic type explicit, we sometimes write $X=X_R$. The flag strata are denoted by $\Flag(X_R)_S$ for $S\subset E_n$. The cone $\Ccal_{K,w}$ and $\Ccal_K$ will be denoted by $\Ccal_{X,S}$ and $\Ccal_X$ respectively. Furthermore, to lighten the notation, we write $\Vcal(\mathbf{k})$ for the automorphic vector bundle $\Vcal_I(\lambda)$ where $\lambda=\chi_{\mathbf{k}}$. Similarly, we write $\Lcal(\mathbf{k})$ instead of $\Vcal_{\flag}(\lambda)$.

\subsection{Definition} \label{subsec-pcones}
For a tuple $(x_1,\dots,x_n)$ and any integer $m\in \ZZ$, we define $x_m$ as the element $x_r$ where $r$ is the unique integer satisfying $r\equiv m \pmod{n}$ and $1\leq r \leq n$ (we say simply that the index $m$ is taken modulo $n$). Let $d\in E_n$ be an integer and $\varepsilon =(\varepsilon_1,\dots,\varepsilon_{n})\in \{\pm 1\}^n$. For all $x=(x_1,\dots,x_n)\in \ZZ^n$, we set:
\begin{equation}
    F^{(d)}_{\varepsilon}(x)\colonequals \sum_{i=0}^{n-1} p^{i} \varepsilon_{i+d}  x_{d+i}
\end{equation}
where the index of $d+i$ is taken modulo $n$. We call $F^{(d)}_{\varepsilon}(x)$ the $p$-expression with starting index $d$ and signs $\varepsilon_1,\dots,\varepsilon_{n}$. Similarly, if $T\subset E_n$ is a subset and $\varepsilon=\delta_T$ is the element of $W$ corresponding to $T$, we write $F^{(d)}_T(x)$ instead of $F^{(d)}_{\varepsilon}(x)$.

\begin{definition}
We say that a cone $C\subset \ZZ^n$ is a $p$-cone if it can be defined by finitely many inequalities of the type 
\begin{equation}
    F^{(d)}_{\varepsilon}(x)\leq 0, \qquad x\in \ZZ^n
\end{equation}
for $d\in E_n$ and $\varepsilon\in \{\pm 1\}^n$.
\end{definition}

To motivate this definition, we consider the cone of partial Hasse invariants $\Ccal_{\pha,S}$ (Definition \ref{def-Cphaw}) and show that it is a $p$-cone. We start with some notation. Let $(x_1,\dots, x_n)$ be a sequence of elements in a ring. For any $i,j\in \ZZ$, we define an element $J_{i, j}(x)$ as follows: If $i\equiv j+1 \pmod{n}$, then we set $J_{i, j}(x)=1$. Otherwise, let $j'$ denote the smallest element $\geq i$ such that $j'\equiv j \pmod{n}$. Put:
\begin{equation}
    J_{i, j}(x) \colonequals \prod_{k=i}^{j'} x_k.
\end{equation}
We then have the following lemma, which is an immediate computation left to the reader:

\begin{lemma}\label{lem-Adj}
For tuples $a=(a_1,\dots,a_n)$ and $b=(b_1,\dots, b_n)$ of elements of a ring, let $M$ be the following matrix
\begin{equation}
  M(a,b) :=  \left(
\begin{matrix}
    a_1 & -b_1 & & \\
    &\ddots&\ddots& \\
    &&\ddots& -b_{n-1} \\
    -b_n&&&a_n 
\end{matrix}
    \right).
\end{equation}
Then, for all $1\leq i,j\leq n$, the $(i,j)$-coefficient of the adjoint matrix $\Adj(M)$ is given by $J_{j+1,i-1}(a) J_{i,j-1}(b)$. Furthermore, the determinant of $M$ is given by
\begin{equation}
    \det(M)=\prod_{i=1}^n a_i - \prod_{i=1}^n b_i.
\end{equation}
\end{lemma}

For $w\in W$, recall that $\Ccal_{\pha,w}$ is defined as the saturation of the cone $h_w(X^*_{+,w}(T))$ where $h_w\colon X^*(T)\to X^*(T)$, $\lambda \mapsto -w\lambda+pw_{0,I}w_0 \sigma^{-1}(\lambda)$. In our case, writing $S$ for the subset corresponding to $w$, we find that the matrix of the map $h_w$ (restricted to $\Lambda=\ZZ^n$) in the standard basis of $\ZZ^n$ is given by
\begin{equation}
  M :=  \left(
\begin{matrix}
    -\delta_S^{(1)} & -p\delta_R^{(1)} & & \\
    &\ddots&\ddots& \\
    &&\ddots& -p \delta_R^{(n-1)} \\
    -p\delta_R^{(n)}&&&-\delta_S^{(n)} 
\end{matrix}
    \right) = M(-\delta_S,\delta_R).
\end{equation}
From Lemma \ref{lem-Adj} above, we deduce the following proposition:
\begin{proposition}\label{prop-pha-cone}
The cone $\Ccal_{\pha,S}$ is a $p$-cone. Specifically, it is given by the following inequalities:
\begin{equation}
(-1)^{|R|} \left(\sum_{1\leq j\leq i-1} J_{j+i,i-1}(-\delta_{S}) J_{i,j-1}(\delta_R) p^{n+j-i} x_j + \sum_{i\leq j \leq  n} J_{j+i,i-1}(-\delta_{S}) J_{i,j-1}(\delta_R) p^{j-i} x_{j}\right) \leq 0
\end{equation}
for all $i\in S$.
\end{proposition}

\subsection{\texorpdfstring{$S$-adapted $p$-cones}{}}

\begin{definition}
Let $S\subset E_n$ be a subset of cardinality $|S|=s$. We say that a $p$-cone $\Ccal$ is $S$-adapted if it is defined by exactly $s$ inequalities whose starting indices are the elements of $S$. In other words, we have
\begin{equation}
    \Ccal = \{ x\in \ZZ^n \ | \ F^{(i)}_{\varepsilon^{(i)}}(x)\leq 0, \ i\in S \}
\end{equation}
for certain elements $\varepsilon^{(i)}\in \{\pm 1 \}^n$ (where $i\in S$).
\end{definition}

It is easy to see from Proposition \ref{prop-pha-cone} that $\Ccal_{\pha,S}$ is an $S$-adapted $p$-cone. For a given subset $S\subset E_n$, in order to define an $S$-adapted $p$-cone, it suffices to give the elements $\varepsilon^{(i)}\in \{\pm 1 \}^n$ for each $i\in S$. Write $\varepsilon^{(i)}=\delta_{T^{(i)}}$ for a certain subset $T^{(i)}\subset E_n$. Thus, an $S$-adapted $p$-cone is uniquely determined by a function
\begin{equation}
    \gamma_{\Ccal}\colon S\to \Pcal(E_n), \quad i\mapsto T^{(i)}.
\end{equation}
This shows that there are exactly $2^{ns}$ such cones.

\begin{definition}
We say that an $S$-adapted $p$-cone $\Ccal$ is homogeneous if $\gamma_{\Ccal}$ is a constant function.
\end{definition}

Therefore, for any subset $S$, the $S$-adapted homogeneous $p$-cones are parametrized by subsets $T\subset E_n$. Specifically, for any subset $T\subset E_n$, there is a unique $S$-adapted homogeneous $p$-cone $\Ccal$ such that $\rho_\Ccal$ is the constant function with value $T$. In general, the partial Hasse invariant cones $\Ccal_{\pha,S}$ are not homogeneous.

%\subsection{Orientation}
%Eventually, the cones that we consider will be related to automorphic forms on EO strata or flag strata. The ampleness of automorphic line bundles determine a positive orientation. Specifically, recall that for any $\lambda\in \Ccal^{\circ}_{\GS}$, there is a bound $N_\lambda$ depending on $\lambda$ such that for $p\geq N_\lambda$ the line bundle $\Vcal_{\flag}(\lambda)$ admits nonzero sections on each flag stratum $\Flag(X_R)_S$. This leads us to the definition of a positively oriented cone. Let $S\subset \{1,\dots,n\}$ be a subset and $\Ccal\subset \ZZ^n$ an $S$-adapted $p$-cone.

\subsection{Chain diagrams} \label{subsec-chain}

%At this point the inequalities defining the various cones $\Ccal_{\pha,w}$ of partial Hasse invariants of flag strata seem quite difficult to understand. We will give a more visual way to determine these inequalities, using the notion of chain diagrams.

In this section, we define certain diagrams as a visual aid in order to represent $p$-cones. We define the standard chain diagram $\Gamma_n$ (of length $n$) as a circular diagram of $n$ vertices numbered from $1$ to $n$, with connecting arcs between vertex $i$ and vertex $i+1$ (where $n+1$ is the same vertex as $1$). More generally, for given subsets $R\subset E_n$ (the parabolic type) and $S\subset E_n$, we define a diagram $\Gamma_n(R,S)$ as follows:
\medskip
\begin{bulletlist}
    \item We consider a $n$ vertices numbered from $1$ to $n$.
    \item We connect $i$ and $i+1$ with a dotted line whenever $i+1\notin S$ and $i\notin R$.
    \item When $i+1\notin S$ and $i\in R$, we connect $i$ and $i+1$ with a plain line.
\end{bulletlist}
\medskip
\noindent When $i+1\in S$, the points $i$ and $i+1$ are never connected. To help visualize the sets $S,R$, we draw a circle (resp. a square) around the vertices corresponding to $S$ (resp. $R$). We call the diagram $\Gamma_n(R,S)$ thus constructed a chain diagram of length $n$, parabolic type $R$ and stratum $S$. This diagram is simply a useful visual way to encode the datum of two subsets $S,R$ of $E_n$. For example, the figure below is the case $n=8$, $R=\{1,3\}$ and $S=\{3,6\}$.

\begin{figure}[H]
    \centering
\includegraphics[width=6.5cm]{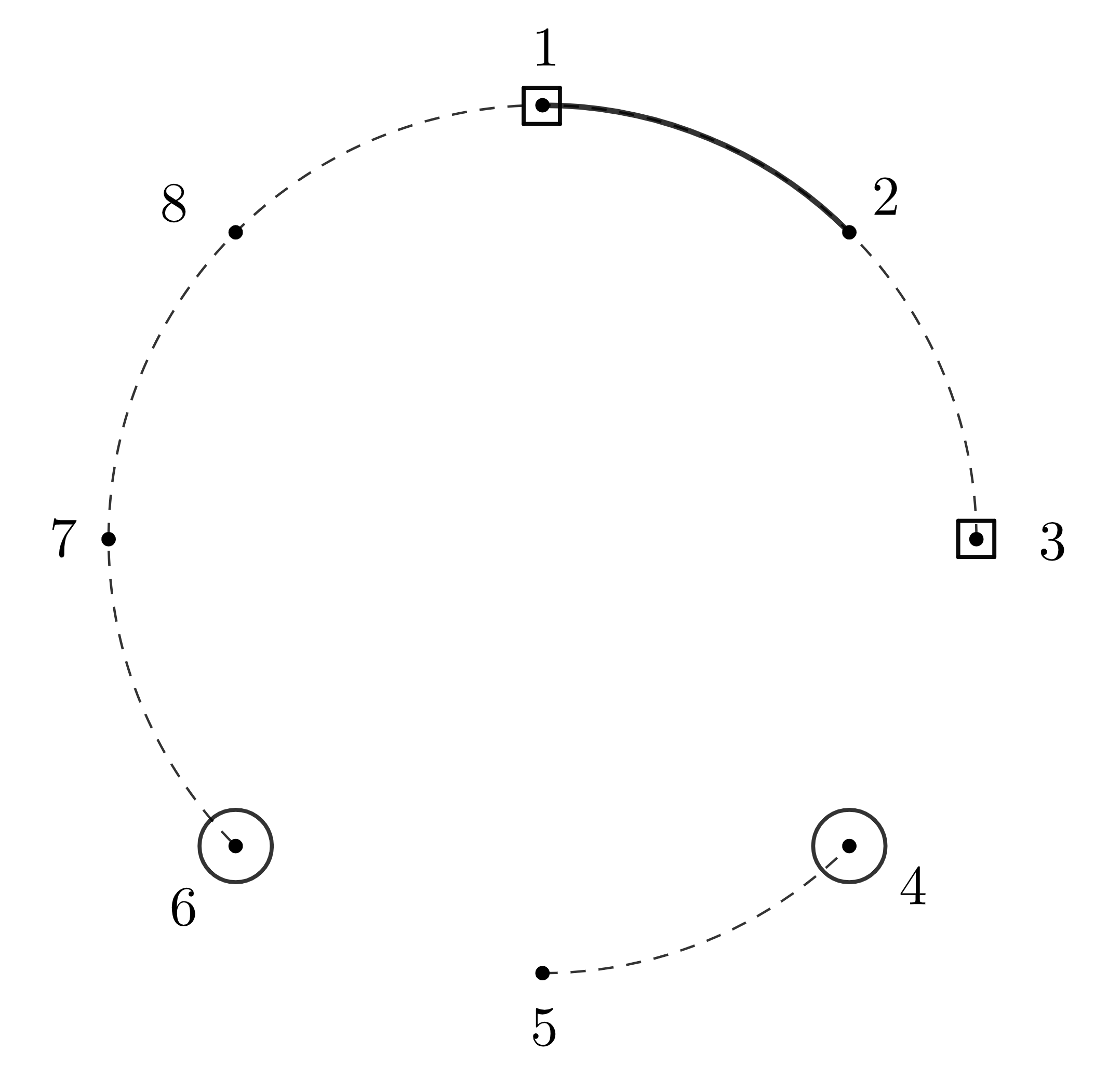}
    \caption{The chain diagram of parabolic type $\{1,3\}$ for the stratum $\{4,6\}$}
    \label{figdiag}
\end{figure}

Eventually, we will show that the vanishing of the $0$th cohomology groups on the flag stratum $\Flag(X_R)_S$ is controlled by the diagram $\Gamma_{n}(R,S)$.

\begin{definition}
Let $\Gamma=\Gamma_n(R,S)$ be the chain diagram of parabolic type $R$ and stratum $S$. We define the following:
\begin{definitionlist}
\item We say that a subset $C \subset E_n$ is connected in $\Gamma$ if any two points of $C$ are connected by a finite sequence of (dotted or plain) lines.
\item We say that $C$ is a connected component of $\Gamma$ if it is maximal among the connected subsets of $\Gamma$.
\item Suppose that $\Gamma$ is disconnected and let $C$ be a connected component of $\Gamma$. We define the head (resp. the tail) of $C$ as the unique element $x\in C$ such that $x+1\notin C$ (resp. $x-1\notin C$).
\end{definitionlist}
\end{definition}
Note that by definition, the tail of a connected component $C$ is the unique element of $S$ that lies in $C$. For example, the connected components of the diagram $\Gamma$ in Figure \ref{figdiag} are $C_1=\{4,5\}$ and $C_2=\{6,7,8,1,2,3\}$. The heads of $C_1$, $C_2$ are respectively $5$ and $3$. The tails of $C_1$, $C_2$ are respectively $4$ and $6$.

\subsection{\texorpdfstring{Admissible $p$-cones}{}} \label{subsec-admiss}

We continue our classification of $p$-cones by defining admissible $p$-cones below. The condition defining this notion is related to the theory of partial Hasse invariants explained in section \ref{part-Hasse-sec}. Again, fix a parabolic type $R\subset E_n$ and a stratum $S\subset E_n$. We explained in Remark \ref{rmk-triv-S} that certain line bundles $\Lcal(\mathbf{k})$ are trivial on the Zariski closure $\overline{\Flag}(X_R)_S$. To be specific, for $i\in E_n$ define
\begin{equation}\label{ha-def}
   \ha_{R,S}^{(i)} = \begin{cases}
        -e_i-p\delta^{(i-1)}_R e_{i-1} & \textrm{if }i\in S \\ 
        e_i-p\delta^{(i-1)}_R e_{i-1} & \textrm{if }i\notin S. 
    \end{cases}
\end{equation}
In short, we may write $\ha_{R,S}^{(i)} = \delta_S^{(i)} e_i-p\delta^{(i-1)}_R e_{i-1}$ for any $i\in E_n$. The cone $\Ccal_{\pha,S}$ is the saturation of the cone
\begin{equation}\label{CphaS-eq}
    C_{\pha,S} = \sum_{i\in S} \NN \ \ha_{R,S}^{(i)} \ + \ \sum_{i\notin S} \ZZ \ \ha_{R,S}^{(i)}.
\end{equation}
As we noted in Remark \ref{rmk-triv-S}, the line bundle $\Lcal(\mathbf{k})$ is trivial for any $\mathbf{k}$ in the $\ZZ$-module spanned by the weights $\ha_{R,S}^{(i)}$ for $i\in E_n\setminus S$. Write $\KK_{R,S}$ for the saturation of this $\ZZ$-module inside $\ZZ^n$. It is also a submodule of $\ZZ^n$. It coincides with the set of weights $\mathbf{k}\in \ZZ^n$ where all the inequalities defining $\Ccal_{\pha,S}$ (see Proposition \ref{prop-pha-cone}) vanish. We call the subgroup $\KK_{R,S}$ the kernel of the stratum $S$. When the choice of $R$ is clear, we simply write $\KK_S$ for this subgroup. For any $\mathbf{k}\in \KK_{R,S}$, the line bundle $\Lcal(\mathbf{k})$ is torsion on the Zariski closure $\overline{\Flag}(X_R)_S$. We may rewrite equation \eqref{CphaS-eq} as follows:
\begin{equation}\label{CphaS-eq2}
    C_{\pha,S} = \sum_{i\in S} \NN \ \ha_{R,S}^{(i)} \ + \KK_{R,S}.
\end{equation}

Since we want to study the various cones of mod $p$ automorphic forms on the various strata, it is natural to impose the condition that the kernel of $S$ is contained in the cone that we consider, since elements of $\KK_S$ correspond to torsion line bundles. This leads to the definition of admissibility of an $S$-adapted $p$-cone that we explain below.

\begin{definition}\label{def-admis}
Let $\Gamma_n(R,S)$ be the chain diagram of parabolic type $R$ and stratum $S$, and let $T\subset E_n$ be a subset.
\begin{definitionlist}
\item We say that $T$ is admissible (or $\Gamma$-admissible) if it satisfies the following:
\begin{bulletlist}
    \item If two vertices are connected by a dotted line, then one of them lies in $T$ and the other one lies in $E_n\setminus T$.
    \item If two vertices are connected by a plain line, then either both vertices lie in $T$ or both vertices lie in $E_n\setminus T$.
\end{bulletlist}
We denote by $\Adm(\Gamma)\subset \Pcal(E_n)$ the set of all $\Gamma$-admissible subsets.
\item Let $\Ccal\subset \ZZ^n$ be an $S$-adapted $p$-cone, and let $\gamma_{\Ccal}\colon S\to  \Pcal(E_n)$ be the corresponding function. We say that $\Ccal$ is admissible if $\gamma_{\Ccal}$ takes values in the subset $\Adm(\Gamma)$.
\end{definitionlist}
\end{definition}

For a chain diagram $\Gamma$, write $\pi_0(\Gamma)$ for the set of connected components. Then an admissible subset $T\subset E_n$ is entirely determined by the knowledge of the intersection $T\cap S$. Indeed, consider the tail $x$ of a connected component $C\subset \Gamma$. Then, if we know whether or not $x\in T$, then we can determine inductively which elements of $C$ lie in $T$ by using Definition \ref{def-admis}(a). We deduce that there are exactly $2^{|S|}$ admissible subsets in $E_n$. Similarly, an admissible $p$-cone is determined by a function $\gamma_{\Ccal}\colon S\to \Adm(\Gamma)$, so there are exactly $2^{|S|\times |S|}$ such cones. Finally, the number of homogeneous admissible $S$-adapted $p$-cones is also $2^{|S|}$, since $\gamma_{\Ccal}$ is constant in that case. We end this section by making the link between the above discussion regarding the kernel $\KK_S$ and the definition of admissibility:

\begin{lemma}\label{admiss-lem}
Let $\Ccal$ be an $S$-adapted $p$-cone. Then $\Ccal$ is admissible if and only if it satisfies the containment $\KK_{S}\subset \Ccal$.
\end{lemma}

\begin{proof}
For $d\in S$, denote simply by $F_{\Ccal}^{(d)}(x)\leq 0$ the $p$-expression with starting index $d$ corresponding to the subset $\gamma_{\Ccal}(d)$, i.e. the inequality $F_{\gamma_\Ccal(d)}^{(d)}(x)\leq 0$. Then, the condition that $\Ccal$ is admissible simply means that for all $d\in S$ and all $x\in \KK_S$, one has $F_{\Ccal}^{(d)}(x)=0$. Hence, when $\Ccal$ is admissible it contains $\KK_{S}$. Conversely, assume $\KK_S\subset \Ccal$ and let $x\in \KK_S$. Since $\KK_S$ is a subgroup, $-x$ also lies in $\KK_S$, which implies that for any $d\in S$ both $F_{\Ccal}^{(d)}(x)$ and $F_{\Ccal}^{(d)}(-x)$ are non-positive. Hence $F_{\Ccal}^{(d)}(x)=0$ and we deduce that $\Ccal$ is admissible.
\end{proof}

\subsection{Positivity}
Next, we consider another natural condition on $p$-cones that stems from the geometry of Shimura varieties. Recall (section \ref{sec-asymp}) that when $p$ goes to infinity, ampleness considerations imply that the open Griffiths--Schmid cone $\Ccal_{\GS}^{\circ}$ is contained "at the limit" in all the cones $\Ccal_{K,w}$. This imposes a "positive direction" for characters. In our case, the cone $\Ccal_{\GS}^{\circ}$ (intersected with $\Lambda$) is given by

\begin{equation}
\Ccal_{\GSsf}^{\circ}=\left\{ (x_1,\dots, x_n)\in \ZZ^n \ \relmiddle| \ 
\parbox{4cm}{
$x_i > 0 \ \textrm{ for }i\in R, \\
x_i < 0 \ \textrm{ for } i\in E_n\setminus R$}
\right\}.
\end{equation}
Taking inspiration from this discussion, we introduce the following positivity condition for $S$-adapted $p$-cones. Let $\Gamma$ be the chain diagram of parabolic type $R$ and stratum $S$. Let $\Ccal\subset \ZZ^n$ be an $S$-adapted $p$-cone and write $\gamma_{\Ccal}\colon S\to \Pcal(E_n)$ for the associated function.

\begin{definition} \label{def-posit} \ 
    \begin{definitionlist}
        \item Let $i\in E_n$ be an element. We say that a subset $T\subset E_n$ is $i$-positive if it satisfies
        \begin{equation}\label{cond-pos}
    i-1\in T \ \Longleftrightarrow \ i-1\in R.
\end{equation}
\item We say that $\Ccal$ is positive if for all $i\in S$, the set $\gamma_{\Ccal}(i)$ is $i$-positive.
    \end{definitionlist}
\end{definition}
Let $\Ccal\subset \ZZ^n$ be any $p$-cone (not necessarily $S$-adapted), defined by some $p$-expression $F^{(d_k)}_{\varepsilon^{(k)}}(x)\leq 0$ (for $k=1,\dots,m$) where $d_k\in E_n$ and $\varepsilon^{(k)}\in \{\pm 1\}^n$. We define the limit cone $\Ccal_{\infty}$ as the "limit" of the cone $\Ccal$ when we let $p$ go to infinity. It is the cone obtained by retaining the dominant coefficients of $p$ in each $p$-expression  $F^{(d_k)}_{\varepsilon^{(k)}}(x)$. Specifically, we write $\varepsilon^{(k)}=(\varepsilon^{(k)}_1,\dots , \varepsilon^{(k)}_n)$ for each $k$ and we put
\begin{equation}
    \Ccal_{\infty} = \{ x\in \ZZ^n \ | \ \varepsilon^{(k)}_{d_k-1} x_{d_k-1} \leq 0, \ k=1,\dots, m \}.
\end{equation}

\begin{proposition}
Let $\Ccal$ be an $S$-adapted $p$-cone. Then $\Ccal$ is positive if and only if it satisfies the containment $\Ccal_{\GS}^{\circ}\subset \Ccal_{\infty}$.
\end{proposition}

\begin{proof}
For each $i\in S$, write $\Ccal_i$ for the cone defined by the inequality $F^{(i)}_{T^{(i)}}(x)\leq 0$ where $T^{(i)}=\gamma_\Ccal(i)$. Thus, $\Ccal$ is the intersection of the cones $\Ccal_i$ for $i\in S$. Similarly, $\Ccal_{\infty}$ is the intersection of the cones $\Ccal_{i,\infty}$ for $i\in S$. Therefore, the inclusion $\Ccal_{\GS}^{\circ}\subset \Ccal_{\infty}$ is equivalent to the condition that for any $i\in S$, the cone $\Ccal_{\GS}^{\circ}$ is contained in $\Ccal_{i,\infty}$ for all $i\in S$. The coefficient in front of $p^{n-1}$ in the $p$-expression $F^{(i)}_{T^{(i)}}(x)$ is $-1$ when $i-1\in T^{(i)}$ and is $1$ when $i-1\notin T^{(i)}$. The result follows easily.
\end{proof}

For example, one can check easily that the $S$-adapted $p$-cone $\Ccal_{\pha,S}$ is positive.

\subsection{\texorpdfstring{Partial Hasse invariants and $p$-cones}{}}

As we saw in section \ref{subsec-admiss}, admissibility characterizes when an $S$-adapted $p$-cone $\Ccal$ contains the kernel $\KK_S$ of a stratum $S$. We explain here a stronger notion of admissibility which corresponds to the stronger containment $\Ccal_{\pha,S}\subset \Ccal$ where $\Ccal_{\pha,S}$ is the cone of partial Hasse invariants for the stratum $S$.

We say that a subset $T\subset E_n$ is Hasse-admissible if it is admissible and for all $i\in S \setminus T$, the following condition holds:
\begin{equation}\label{cond-strong}
    i-1\in T \ \Longleftrightarrow \ i-1\in R.
\end{equation}
Using the terminology of Definition \ref{def-posit}, the above condition means that $T$ is $i$-positive with respect to all the elements in $S\setminus T$. Let $\Ccal$ be an $S$-adapted $p$-cone. We say that $\Ccal$ is Hasse-admissible if the corresponding function $\gamma_{\Ccal}\colon S\to \Pcal(E_n)$ takes values in Hasse-admissible subsets. Hence, if $\Ccal$ is Hasse-admissible and positive, the set $\rho_\Ccal(i)$ is $j$-positive for all $j\in (S\setminus T)\cup\{i\}$.

\begin{proposition}
Let $\Ccal$ be an $S$-adapted $p$-cone. The containment $\Ccal_{\pha,S}\subset \Ccal$ holds if and only if $\Ccal$ is Hasse-admissible and positive.
\end{proposition}
\begin{proof}
Let $T\subset E_n$ be an admissible subset. Let $i\in S$ and write $\Ccal^{(i)}_T$ for the cone in $\ZZ^n$ defined by the inequality $F_T^{(i)}(x)\leq 0$. We will show the following: the inclusion $\Ccal_{\pha,S}\subset \Ccal^{(i)}_T$ holds if and only if $T$ is $j$-positive for all $j\in (S\setminus T)\cup\{i\}$.

Write $\delta_T=(\delta^{(1)}_T,\dots,\delta^{(n)}_T)$ for the characteristic function of $T$, i.e. $\delta^{(j)}_T = -1$ for $j\in T$ and $\delta^{(j)}_T = 1$ for $j\notin T$. Recall that we defined the weights $\ha_{R,S}^{(j)}$ for all $j\in E_n$ in \eqref{ha-def}. For $j\in S$ we have $\ha_{R,S}^{(j)} = e_j-\delta^{(j-1)}_R p e_{j-1}$. Since $T$ is admissible, we already know that $\KK_{R,S}\subset \Ccal^{(i)}_T$. It follows that the inclusion $\Ccal_{\pha,S}\subset \Ccal^{(i)}_T$ is satisfied if and only if $F_T^{(i)}(\ha_{R,S}^{(j)})\leq 0$ for all $j\in S$. For $j\neq i$, this inequality amounts to
\begin{equation}
    \delta_{R}^{(j-1)} \delta_{T}^{(j-1)} + \delta_{T}^{(j)} \leq 0,
\end{equation}
which is equivalent to the condition: If $j\notin T$, then $j-1\in R \ \Longleftrightarrow \ j-1 \in T$. On the other hand, when $j=i$, the inequality $F_T^{(i)}(\ha_{R,S}^{(i)})\leq 0$ amounts to
\begin{equation}\label{eq-lead}
    -p^n \delta_{R}^{(i-1)} \delta_{T}^{(i-1)}  + \delta_{T}^{(i)} \leq 0,
\end{equation}
which is equivalent to $\delta_{R}^{(i-1)} \delta_{T}^{(i-1)}=1$, thus to $i-1\in R \ \Longleftrightarrow \ i-1 \in T$. We have proved that $\Ccal_{\pha,S}\subset \Ccal^{(i)}_T$ holds if and only if $T$ is $j$-positive for all $j\in (S\setminus T)\cup\{i\}$. The result follows immediately.
\end{proof}

We know that $\Ccal_{\pha,S}$ is a positive admissible $S$-adapted $p$-cone. In particular, it is given by a certain function
\begin{equation}\label{rho-pha}
    \rho_{\pha,S}\colon S\to \Pcal(E_n).
\end{equation}
Concretely, this function is given as follows: Write $S^c$ for the complement of $S$ in $E_n$. For any $i\in S$, the set $\rho_{\pha,S}(i)$ is the set of elements $j\in E_n$ satisfying: the number of elements in the sets $R\cap [j,i-1[$ and $S^c\cap ]j,i[$ have different parity. This is simply a reformulation of Proposition \ref{prop-pha-cone}.

\section{Cohomology vanishing}\label{coho-van-sec}

The key observation of this paper is that intersection-sum cones of Shimura varieties of $A_1$-type follow a very clear pattern. For general Hodge-type Shimura varieties, these cones are combinatorially difficult to describe. However, the results of this paper seem to indicate that they have a deeper significance and admit a more conceptual definition.

\subsection{\texorpdfstring{The function $\Phi$}{}}

We fix a parabolic type $R\subset E_n$. We start with the following proposition:

\begin{proposition}\label{prop-unique}
Let $S\subset E_n$ be a subset. There exists a unique positive admissible homogeneous $S$-adapted $p$-cone.
\end{proposition}

\begin{proof}
Recall that a homogeneous $S$-adapted $p$-cone $\Ccal$ can be parametrized by a subset $T\subset E_n$ such that $\gamma_{\Ccal}$ is the constant function with value $T$. The positivity condition implies that for all elements $i\in S$, the subset $T$ is $i$-positive, which means that
\begin{equation}
    i-1\in T \ \Longleftrightarrow \ i-1\in R.
\end{equation}
The heads of the connected components are exactly the elements $i-1$ for $i\in S$. Therefore, the above equivalence determines completely which heads lie in $T$. But then, the admissibility condition implies that we know $T$ completely on each connected component. The result follows.
\end{proof}
Note that the proof of Proposition \ref{prop-unique} gives a way to construct the subset $T$: We first determine which heads of connected component are in $T$, and then we extend $T$ by admissibility. 
\begin{exa}\label{exa-T-compute}
Let us give an example for illustration. Let $n=8$, $R=\{1,3\}$, $S=\{4,6\}$ as in Figure \ref{figdiag}. The associated chain diagram has two connected components, whose heads are respectively $3$ and $5$. Of these two elements, only $3$ lies in $R$. Hence $3\in T$ and $5\notin T$. Using the admissibility, we deduce the rest of elements of $T$. We find in the end:
\begin{equation}
    T=\{3,4,6,8\}.
\end{equation}
Therefore, the unique positive admissible homogeneous $S$-adapted $p$-cone is defined by the inequalities:
\begin{align*}
 &p^5 x_1 + p^6 x_2 - p^7 x_3 - x_4 + p x_5 - p^2 x_6 + p^3 x_7 - p^4 x_8 \leq 0 \\
    &p^3x_1 + p^4x_2 - p^5 x_3 - p^6 x_4 + p^7 x_5 - x_6 + p x_7 - p^2 x_8 \leq 0.
\end{align*}
Recall that the starting indices are imposed by the indices in $S$, and the minus signs are assigned by the indices in $T$.
\end{exa}

For any subset $S\subset E_n$, we write $\Phi_R(S)$ for the subset $T\subset E_n$ corresponding to the unique positive admissible homogeneous $S$-adapted $p$-cone (for simplicity, we omit $n$ from the notation). For instance, in Example \ref{exa-T-compute} above, one has $\Phi_R(\{4,6\})=\{3,4,6,8\}$ for $n=8$ and $R=\{1,3\}$. This defines a function
\begin{equation}
    \Phi_R\colon \Pcal(E_n) \to \Pcal(E_n).
\end{equation}
We write $\Ccal_{R,S}$ for the unique positive, admissible $S$-adapted homogeneous $p$-cone. In other words, it is the $S$-adapted $p$-cone that corresponds to the constant function $\rho\colon S\to \Pcal(E_n)$ with value $\Phi_R(S)$. When the choice of $R$ is clear, we simply write $\Phi(S)$ for the set $\Phi_R(S)$ and $\Ccal_S$ for the cone $\Ccal_{R,S}$.

\subsection{Main result}

Fix a parabolic type $R\subset E_n$. Our main technical result is the following theorem.

\begin{theorem}\label{main-thm-hom}
Let $S\subset E_n$ be a subset. The intersection-sum cone $\Ccal^{\cap,+}_S$ is a positive, admissible $S$-adapted homogeneous $p$-cone. Hence, it coincides with the unique such cone $\Ccal_{R,S}$ afforded by Proposition \ref{prop-unique}. 
\end{theorem}

Note that in general a convex hull of $p$-cones is not a $p$-cone, thus it is not a priori obvious that $\Ccal^{\cap,+}_S$ is defined by $p$-expressions, let alone that it is $S$-adapted and homogeneous. However, once we know that it is an $S$-adapted $p$-cone, it is clear that it must be positive and admissible because it contains $\Ccal_{\pha,S}$ by definition. Theorem \ref{main-thm-hom} shows that the intersection-sum cone is a natural and meaningful construction. It is possible to define these cones for any Shimura varieties (and even independently of Shimura varieties), thus one can hope for a generalization of Theorem \ref{main-thm-hom} to arbitrary reductive groups. Howeveer it is unclear to us what form this generalization would take.

As a consequence of the above theorem, we obtain the following vanishing result. Let $X=X_R$ be a Hodge-type Shimura variety of $A_1$-type and parahoric type $R$.

\begin{theorem}\label{vanish-thm}
Let $S\subset E_n$ be a subset. For any $\mathbf{k}\in \ZZ^n$ such that $\mathbf{k}\notin \Ccal_{R,S}$, we have
\begin{equation}
    H^0(\overline{\Flag}(X_R)_S, \Lcal(\mathbf{k})) = 0.
\end{equation}
\end{theorem}

\begin{proof}
Since $\Ccal_S^{+,\cap} = \Ccal_{R,S}$ by Theorem \ref{main-thm-hom}, this follows immediately from Theorem \ref{thm-sep-syst}.
\end{proof}

The proof of Theorem \ref{main-thm-hom} will occupy a large portion of the remainder of this section. We make a few preparatory remarks in the next paragraphs below.

\subsection{Ordering}\label{order-sec}

We make some natural definitions about the ordering of integers modulo $n$. Let $S\subset E_n$ be a subset. For elements $x,y\in S$, say that $y$ follows $x$ in $S$ if $y$ is the first element of $S$ that appears in the sequence
\begin{equation}\label{seq-x}
    x+1, x+2, x+3, \dots
\end{equation}
(recall that an integer $k$ is identified with the unique element of $E_n$ congruent to $k$ modulo $n$). For three vertices $x,y,z\in E_n$, we say that $z$ is between $x$ and $y$ if $z$ appears before $y$ (or at the same time as $y$) in the sequence \eqref{seq-x}. Write $[x,y]$ for the set of elements that are between $x$ and $y$. Similarly, define $]x,y] \colonequals [x,y] \setminus\{x\}$, $[x,y[ \  \colonequals [x,y] \setminus \{y\}$ and $]x,y[ \ \colonequals [x,y] \setminus\{x,y\}$.

In certain cases we will write $S=\{s_1,\dots, s_r\}$ and define cyclically $s_{i+mr}\colonequals s_i$ for any $m\in \ZZ$ (for example, $s_{r+1}=s_1$). By re-numbering the elements, we may assume that for each $1\leq d\leq r$, the element $s_{d+1}$ follows $s_d$ in the set $S$. Note that we are free to choose which element of $S$ is $s_1$, and this choice then determines uniquely the numbering of all elements.

\subsection{Galois translation}

Recall that the Frobenius $\sigma$ acts on $\ZZ^n$ as follows:
\begin{equation}
    \sigma \cdot (k_1,\dots, k_n) = (k_n, k_1, \dots ,k_{n-1}).
\end{equation}
In particular $\sigma e_i=e_{i+1}$ where $(e_1,\dots, e_n)$ is the standard basis of $\ZZ^n$. Similarly, for $i\in E_n$ define $\sigma(i)\colonequals i+1$ (recall that $n+1$ is identified with $1$). When we consider a subset $S=\{s_1,\dots, s_r\}$ (ordered as in section \ref{order-sec} above), this translation operator $\sigma$ will be useful to reduce to the case $s_1=1$. This will make notation more convenient in certain computations. Note that the operator $\sigma$ acts naturally on all objects that we have defined so far. For example, one sees easily that for any subset $T\subset \{1,\dots,n\}$ and $i\in T$, one has
\begin{equation}
    F^{(i)}_T(x) =  F^{(\sigma i)}_{\sigma T}(\sigma x)
\end{equation}
for all $x=(x_1,\dots, x_n)\in \ZZ^n$. Similarly, we have
\begin{equation}
    \Ccal_{\sigma R, \sigma S} = \sigma(\Ccal_{R,S})
\end{equation}
where $\Ccal_{R,S}$ is the unique positive admissible homogeneous $S$-adapted $p$-cone. This follows from the relation $\Phi_{\sigma R}(\sigma S) = \sigma(\Phi_R(S))$ that is clear from the construction of $\Phi_R(S)$. Similarly, one has $\Ccal_{\pha,\sigma S}=\sigma(\Ccal_{\pha,S})$ and $\Ccal_{\sigma S}^{\cap , +} = \sigma (\Ccal_{ S}^{\cap, +})$ (caution: here the cones $\Ccal_{\pha,\sigma S}$ and $\Ccal_{\sigma S}^{\cap , +}$ must be taken with respect to the parabolic type $\sigma R$).

Note that these natural manipulations involving the action of $\sigma$ would be impossible to carry out on the Shimura variety $X_R$. Indeed, since $P_R$ is not defined over $\FF_p$, the various cones attached to $R$ and $\sigma R$ correspond to different Shimura varieties altogether and hence are a priori unrelated. This freedom is one of the advantages of our group-theoretical approach.

\subsection{Removable elements} \label{sec-remove}

Fix a parabolic type $R$. We omit $R$ from the notation.

\begin{definition} \ 
\begin{definitionlist}
    \item For a subset $S\subset E_n$ and an element $j\in S$, we say that $j$ is removable from $S$ if 
\[\Phi(S)=\Phi(S\setminus\{j\}).\]
\item We say that $S$ is irreducible if none of the elements of $S$ is removable from $S$.
\end{definitionlist}
\end{definition}

One has the obvious but useful lemma below:
\begin{lemma}\label{lemma-remove}
Assume that $j\in S$ is a removable element of $S$ and write $T\colonequals \Phi(S)$. Then the cone $\Ccal_{S\setminus\{j\}}$ satisfies the inequality 
\[
F^{(i)}_{T}(x)\leq 0
\]
for any $i\in S\setminus\{j\}$.    
\end{lemma}

\begin{proof}
Indeed, since $\Phi(S)=\Phi(S\setminus\{j\})$ and $i\neq j$, one of the inequalities defining the cone $\Ccal_{S\setminus\{ j\}}$ is precisely $F^{(i)}_{T}(x)\leq 0$.
\end{proof}

\subsection{Invertible elements}

For a cone $\Ccal\subset \ZZ^n$, denote by $\KK(\Ccal)\subset \Ccal$ the set of invertible elements, i.e. the elements $\lambda\in \Ccal$ such that $-\lambda\in \Ccal$. The set $\KK(\Ccal)$ is the largest subgroup contained in $\Ccal$. For example, $\KK(\Ccal_{\pha,S})=\KK_S$.

\begin{definition}
Let $\Ccal\subset \ZZ^n$ be a cone, and $\Ccal'\subset \Ccal$ a subcone. We say that $\Ccal'$ is saturated in $\Ccal$ if for all $n\geq 1$ and all $x\in \Ccal$, one has $nx\in \Ccal' \Longrightarrow x\in \Ccal'$.
\end{definition}

In particular when $\Ccal=\ZZ^n$, we have the notion of a saturated subcone of $\ZZ^n$. If $\Ccal$ is saturated in $\ZZ^n$ and $\Ccal'$ is saturated in $\Ccal$, then $\Ccal'$ is also saturated in $\ZZ^n$. A subcone $\Ccal\subset \ZZ^n$ defined by linear inequalities is always saturated in $\ZZ^n$ (thus all $p$-cones are saturated in $\ZZ^n$). It is clear that $\KK(\Ccal)$ is always saturated in $\Ccal$. It satisfies even a stronger property: if $x,y\in \Ccal$ and $x+y\in \KK(\Ccal)$, then both $x,y$ are in $\KK(\Ccal)$.

\begin{lemma}\label{KC-lemma}
Let $\Ccal\subset \ZZ^n$ be a cone defined by $r$ linearly independent inequalities
\begin{equation}
    f_i(x)\leq 0 , \quad x\in \ZZ^n, \quad i=1,\dots, r.
\end{equation}
By this, we mean that the linear forms $f_i\colon \QQ^n\to \QQ$ are linearly independent over $\QQ$. Then $\KK(\Ccal)$ is a free module of rank $n-r$, and coincides with the set of $x\in \ZZ^n$ such that $f_i(x)=0$ for all $i=1,\dots,r$.
\end{lemma}

\begin{proof}
It is clear that $\KK(\Ccal)$ coincides with the set of $x\in \ZZ^n$ such that $f_i(x)=0$ for all $i=1,\dots, r$. Since the $f_i$ are linearly independent, they vanish on a codimension $r$ subspace of $\QQ^n$. Hence, the rank of $\KK(\Ccal)$ is $n-r$.
\end{proof}

\begin{lemma}\label{Sadap-lin-indep}
Let $S\subset E_n$ be a subset and $\Ccal$ an $S$-adapted $p$-cone. Then, the $|S|$ inequalities defining $\Ccal_S$ are linearly independent.
\end{lemma}

\begin{proof}
    $\Ccal$ is defined by inequalities of the form $F^{(i)}_{T^{(i)}}(x)\leq 0$ for $i\in S$ and certain subsets $T^{(i)}\subset E_n$. It is clear that the forms $F^{(i)}_{T^{(i)}}\colon \ZZ^n\to \ZZ$ are linearly independent when we reduce them modulo $p$ (because the starting indices are pairwise distinct). Hence, they are also linearly independent over $\QQ$.
\end{proof}

\begin{corollary} \label{cor-S-adap-KC}
Let $S\subset E_n$ be a subset and $\Ccal$ an $S$-adapted $p$-cone. Then $\KK(\Ccal)$ is a free module of rank $n-|S|$ which is saturated in $\ZZ^n$.
\end{corollary}

We deduce that if $\Ccal'\subset \Ccal\subset \ZZ^n$ are two saturated subcones of $\ZZ^n$, then one has $\KK(\Ccal')\subset \KK(\Ccal)$ and the quotient $\KK(\Ccal)/\KK(\Ccal')$ is free of finite rank.

\begin{proposition}\label{lem-KS-homog}
Let $\Ccal\subset \ZZ^n$ be an $S$-adapted admissible $p$-cone. Then one has
\[\KK(\Ccal) = \KK_S.\]    
\end{proposition}

\begin{proof}
Since $\Ccal$ is admissible, it contains $\KK_S=\KK(\Ccal_{\pha,S})$ by Lemma \ref{admiss-lem}. Hence we have $\KK(\Ccal_{\pha,S})\subset \KK(\Ccal)$. Since both $\Ccal_{\pha,S}$ and $\Ccal$ are $S$-adapted, we know by Corollary \ref{cor-S-adap-KC} that $\KK(\Ccal_{\pha,S})$ and $\KK(\Ccal)$ are free modules of rank $n-|S|$. Furthermore, since the quotient $\KK(\Ccal) / \KK(\Ccal_{\pha,S})$ is free, we deduce that $\KK(\Ccal) = \KK(\Ccal_{\pha,S})$.
\end{proof}

\subsection{System of generators}

\begin{definition}
Let $\Ccal\subset \ZZ^n$ be a subcone of $\ZZ^n$ and $\lambda_1,\dots, \lambda_r\in \Ccal$.
\begin{definitionlist}
 \item We say that $\lambda_1,\dots,\lambda_r$ is a system of $\QQ_{\geq 0}$-generators of $\Ccal$ if any element of $\Ccal$ can be written as a linear combination of the $\lambda_i$ with non-negative rational coefficients.
  \item We say that $\lambda_1,\dots,\lambda_r$ is a system of $\QQ_{\geq 0}$-generators of $\Ccal$ modulo kernel if any element of $\Ccal$ can be written as $x+y$ where $x$ is a linear combination of the $\lambda_i$ with nonnegative rational coefficients and $y\in \KK(\Ccal)$.
\end{definitionlist}
\end{definition}

Fix a parabolic type $R\subset E_n$. Let $S\subset E_n$ be a subset and let $\Ccal_{\pha,S}$ be the cone of partial Hasse invariants for $S$. Recall that $\Ccal_{\pha,S}$ is the saturation of the cone
\begin{equation}
    C_{\pha,S} = \sum_{i\in S} \NN \ \ha_{R,S}^{(i)} \ + \ \sum_{i\notin S} \ZZ \ \ha_{R,S}^{(i)}
\end{equation}
where $\ha_{R,S}^{(i)} = -\delta^{(i)}_S e_i - p \delta_R^{(i-1)} e_{i-1}$. The kernel $\KK(\Ccal_{\pha,S})=\KK_S$ is generated as a group by the elements $\ha^{(i)}_{R,S}$ for $i\notin S$. Thus, the elements $\ha_{R,S}^{(i)} =  e_i-p\delta_R^{(i-1)} e_{i-1}$ for $i\in S$ form a system of $\QQ_{\geq 0}$-generators of $\Ccal_{\pha,S}$ modulo kernel.

Let $\Ccal_{R,S}$ be the unique positive admissible homogeneous $S$-adapted $p$-cone. We seek a system of $\QQ_{\geq 0}$-generators modulo kernel for $\Ccal_{R,S}$. Write $T\colonequals \Phi_R(S)$ and define for all $i\in S$ the vector:
\begin{equation}
\gen_{R,S}^{(i)} = \delta_T^{(i)} e_i- p\delta_R^{(i-1)} e_{i-1}.
\end{equation}
Note that when $i \in S\setminus T$ we have $\delta_T^{(i)}=1$ and hence $\gen_{R,S}^{(i)} = \ha_{R,S}^{(i)}$. Also, since the cone $\Ccal_{R,S}$ is positive, we have $\delta^{(i-1)}_T=\delta^{(i-1)}_R$, so we can also write $\gen_{R,S}^{(i)} = \delta_T^{(i)} e_i-p\delta_T^{(i-1)} e_{i-1}$.

\begin{proposition}
The vectors $\gen_{R,S}^{(i)}$ (for $i\in S$) form a system of $\QQ_{\geq 0}$-generators modulo kernel of $\Ccal_{R,S}$.
\end{proposition}

\begin{proof}
The cone $\Ccal_{R,S}$ is defined by the $|S|$ inequalities $F^{(i)}_T(x)\leq 0$ where $i\in S$ and $T=\Phi_R(S)$. Hence, it suffices to show that $F^{(i)}_T(\gen_{R,S}^{(j)})$ is zero for distinct elements $i,j\in S$ and is negative for $i=j$. Assume first that $i\neq j$. Looking at the expression of $F^{(i)}_T(x)$ and the fact that $i\neq j$, one sees easily that $F^{(i)}_T(x)$ vanishes when $x=\delta^{(j)}_T e_j - p \delta^{(j-1)}_T e_{j-1}$, hence $F^{(i)}_T(\gen_{R,S}^{(j)})=0$. Finally, in the case $i=j$ we find that $F^{(i)}_T(\gen_{R,S}^{(i)}) = -(p^n-1)$, which terminates the proof.
\end{proof}

We deduce that $\Ccal_{R,S}$ is the saturation in $\ZZ^n$ of the following cone:
\begin{equation}\label{CS-eq}
\sum_{i\in S} \NN \ \gen_{R,S}^{(i)} \ + \KK_{R,S}.
\end{equation}

\subsection{Proof of Theorem \ref{main-thm-hom}}
This section contains the technical computations necessary to prove Theorem \ref{main-thm-hom}. We fix an integer $n\geq 1$ and a parabolic type $R\subset E_n$ throughout and omit them from all notation. In particular, we write $\Phi(S)$ instead of $\Phi_R(S)$. We will show by induction that for all strata $S\subset E_n$, the intersection-sum cone $\Ccal^{\cap,+}_S$ coincides with $\Ccal_S$ (the unique positive, admissible $S$-adapted homogeneous $p$-cone).

For a stratum such that $|S|=1$, the intersection-sum $\Ccal_{S}^{\cap,+}$ is defined to be $\Ccal_{\pha,S}$. We already know that this is an admissible, positive $S$-adapted $p$-cone. Furthermore, since $|S|=1$ it is obviously homogeneous. Therefore, we may assume $|S|\geq 2$ and prove the result by induction. Recall that $\Ccal_S$ is defined by the $|S|$ inequalities
\begin{equation}
    F^{(i)}_{T}(x)\leq 0, \qquad x\in \ZZ^n
\end{equation}
where $T=\Phi(S)$ and $i$ varies in $S$. We need to show that $\Ccal^{\cap,+}_S=\Ccal_S$. Write 
\begin{equation}\label{S-subset}
    S=\{s_1,\dots, s_r\}
\end{equation}
where $r\colonequals |S|$ and the elements are ordered as in section \ref{order-sec}. Consider all lower neighbors of $S$, namely the subsets of the form $S_k\colonequals S\setminus \{s_k\}$ for $1\leq k\leq r$. By the induction hypothesis, we know that the intersection-sum cone of $S_k$ coincides with $\Ccal_{S_k}$, the unique positive admissible homogeneous $S_k$-adapted $p$-cone (where all these notions are taken with respect to the subset $S_k$). We thus have:
\begin{equation}
    \Ccal^{\cap,+}_S = \Ccal_{\pha, S} \ +_{\rm sat} \ \bigcap_{k\in S} \Ccal_{S_k}.
\end{equation}
We need to prove the two inclusions $\Ccal^{\cap,+}_S\subset \Ccal_S$ and  $\Ccal_S\subset \Ccal^{\cap,+}_S$.

\subsubsection{\texorpdfstring{Proof of the inclusion $\Ccal^{\cap,+}_S\subset \Ccal_S$}{}}
This is the more interesting of the two inclusions, as it gives an upper bound for the cone $\Ccal^{\cap,+}_S$, hence also one for the cone $\Ccal_{K,S}$ by Theorem \ref{thm-sep-syst}. Since $\Ccal_S$ is positive and admissible, we already know that $\Ccal_{\pha,S}\subset \Ccal_S$. Thus, it suffices to show that $\bigcap_{1\leq k \leq r} \Ccal_{S_k}\subset \Ccal_S$. We write 
\[\Ccal^{\cap}\colonequals \bigcap_{1\leq k \leq r} \Ccal_{S_k}.\]
We must prove that for each $i\in S$, the cone $\Ccal^{\cap}$ satisfies the inequality $F_T^{(i)}(x)\leq 0$ for $T=\Phi(S)$. Suppose that we have shown: For any $R$ and any $S$ such that $1\in S$, the cone $\Ccal^{\cap}$ satisfies $F_T^{(1)}(x)\leq 0$. Then, let $S,R$ be arbitrary and take $i\in S$. Using the translation operator $\sigma$ and applying the assumption to the sets $\sigma^{-(i-1)} S$, $\sigma^{-(i-1)} R$, we deduce that $\Ccal^{\cap}$ satisfies $F_T^{(i)}(x)\leq 0$. Therefore, we may assume that $s_1=1$ and show that $\Ccal^{\cap}$ satisfies $F_T^{(1)}(x)\leq 0$. In this case, we can write the elements of $S$ as $1=s_1<s_2<\dots < s_r \leq n$ for $r=|S|$.

The easiest situation is when some element $s_k$ ($k\neq 1$) is removable from $S$. As we explained in Lemma \ref{lemma-remove}, in this case the cone $\Ccal_{S_k}$ (where $S_k=S\setminus\{s_k\}$) satisfies the inequality $F_T^{(s_1)}(x)\leq 0$. In particular, the cone $\Ccal^{\cap}$ also satisfies this inequality. Therefore, we are reduced to the situation when none of the elements $s_2,\dots, s_r$ is removable from $S$. We will assume that this is the case from now on. The proof then continues by considering two cases: The case when $s_1$ is removable from $S$ and the case when it is not.

Write $\Phi(S)=T$ and $\Phi(S_k)=T_k$ for $k=1,\dots, r$. We will show that there exist positive integers $\lambda, \lambda_1, \dots, \lambda_r$ satisfying
\begin{equation}\label{Fs1-relation}
   \lambda  F^{(s_1)}_T(x) = \sum_{k=1}^r \lambda_k F^{(s_{k})}_{T_{k+1}}(x)
\end{equation}
for all $x=(x_1,\dots, x_n)\in \ZZ^n$. Since $\Ccal^{\cap}$ satisfies all inequalities $F^{(s_{k})}_{T_{k+1}}(x)\leq 0$ for $k=1,\dots,r$, this will show that it also satisfies $F^{(s_1)}_T(x)\leq 0$. 

\paragraph{The case when $s_1$ is removable} In this case, we define $\lambda, \lambda_1, \dots, \lambda_r$ as follows: set $\lambda\colonequals p (p^n+1)^{r-1}$ and for $1\leq k \leq r-1$ put
\begin{equation}
    \lambda_k \colonequals 2^{k-1} (p^n-1) p^{s_k} (p^n+1)^{r-k-1}.
\end{equation}
For $k=r$, put
\begin{equation}
    \lambda_r\colonequals 2^{r-1} p^{s_r}.
\end{equation}
It is now a tedious computation to check the equality \eqref{Fs1-relation}. Since this computation is the key to the main theorem \ref{main-thm-hom}, we carry it out in details below. We can write
\begin{equation}
   F^{(s_{k})}_{T_{k+1}}(x) = \sum_{i=0}^{n-1} p^i \delta_{T_{k+1}}^{(i+s_k)} x_{i+s_k}.
\end{equation}
Let $\Pi(x)$ denote the right-hand side of \eqref{Fs1-relation}. We find:
\begin{align*}
  \Pi(x) &= (p^n-1)\left( \sum_{k=1}^{r-1} 2^{k-1} p^{s_k} (p^n+1)^{r-k-1} F_{T_{k+1}}^{(s_k)}(x)  \right) + 2^{r-1}p^{s_r} F_{T_1}^{(s_r)}(x) \\
&= (p^n-1)\left( \sum_{k=1}^{r-1} 2^{k-1} p^{s_k} (p^n+1)^{r-k-1}  \sum_{i=0}^{n-1} p^i \delta_{T_{k+1}}^{(i+s_k)} x_{i+s_k} 
 \right) + 2^{r-1}p^{s_r} \sum_{i=0}^{n-1} p^i \delta_{T_{1}}^{(i+s_r)} x_{i+s_r}.  \\  
\end{align*}
We want to find the coefficient in front of the variable $x_m$ (for $1\leq m \leq n$) in the above expression. Note that $x_m$ can appear in two different ways in this expression because one can have $i+s_k =i' +s_{k'}+n$ for different values of $i,k,i',k'$. Taking this into account, we filter $m$ with respect to the intervals $[s_d, s_{d+1}[$. Recall that when $d=r$, we have $s_{r+1}=s_1=1$ and the interval $[s_r, s_1[$ consists of the integers from $s_r$ to $n$. For $m\in [s_{d},s_{d+1}[$ and $d\leq r-1$, the coefficient in front of $x_m$ in $\Pi(x)$ is
\begin{equation*}
c_m\colonequals (p^n-1)p^m \left( \sum_{k=1}^d 2^{k-1}(p^n+1)^{r-k-1} \delta_{T_{k+1}}^{m} + p^{n} \sum_{k=d+1}^{r-1} 2^{k-1} (p^n+1)^{r-k-1} \delta_{T_{k+1}}^{m} \right) + 2^{r-1} p^{m+n} \delta_{T_{1}}^{m}.
\end{equation*}
For $d=r$, the coefficient in front of $x_m$ in $\Pi(x)$ is
\begin{equation*}
c_m\colonequals (p^n-1)p^m \left( \sum_{k=1}^{r-1} 2^{k-1}(p^n+1)^{r-k-1} \delta_{T_{k+1}}^{m}\right) + 2^{r-1} p^{m} \delta_{T_{1}}^{m}.
\end{equation*}
By assumption, none of the elements $s_k$ (for $k\neq 1$) is removable from $S$. Note that when we remove $s_k$ from $S$, only the elements of $[s_{k-1},s_k[$ are affected. Thus, $T\cap [s_{i-1},s_i[$ and $T_i\cap [s_{i-1},s_i[$ coincide for all $i\neq k$ and are complementary for $i=k$. We deduce that $\delta_{T_{k+1}}^{(m)}=\delta_{T}^{(m)}$ for $k\neq d$ and $\delta_{T_{d+1}}^{(m)}=-\delta_{T}^{(m)}$. Using this and the formula for the sum of the geometric sequence $2^{k-1}(p^n+1)^{r-k-1} = \frac{(p^n+1)^{r-1}}{2} \left(\frac{2}{p^n+1}\right)^k$, we find for $1\leq d\leq r-1$:
\begin{align*}
c_m &= (p^n-1)p^m \delta_{T}^{(m)}  \Bigg( \frac{(p^n+1)^{r-1}-2^{d-1}(p^n+1)^{r-d}}{p^n-1}
-  2^{d-1}(p^n+1)^{r-d-1} \\
& \qquad \qquad  + \frac{2^d p^n}{p^n-1} ((p^n+1)^{r-d-1}-2^{r-d-1}) \Bigg) + 2^{r-1} p^{m+n} \delta_{T}^{m}  \\
& = p^m \delta_{T}^{(m)}  \bigg((p^n+1)^{r-1} - (p^n+1)^{r-d-1} ( 2^{d-1} (p^n+1) + 2^{d-1} (p^n-1) -2^d p^n) - 2^{r-1} p^n +2^{r-1}p^n\bigg) \\
& = p^m (p^n+1)^{r-1}  \delta_{T}^{(m)} = \lambda p^{m-1}\delta^{(m)}_T.
\end{align*}
Thus, $c_m$ coincides with the coefficient of $x_m$ in $\lambda F_T^{(s_1)}(x)$. In the above computation, we have not used the assumption that $s_1$ is removable. This assumption is only necessary for the case $d=r$. In this case, we have $\delta^{(m)}_{T_1} = \delta_T^{(m)}$. Therefore, a similar computation yields:
\begin{align*}
c_m &= (p^n-1)p^m \left( \sum_{k=1}^{r-1} 2^{k-1}(p^n+1)^{r-k-1} \delta_{T_{k+1}}^{m}\right) + 2^{r-1} p^{m} \delta_{T_{1}}^{m} \\
& = p^m \delta_{T}^{(m)}  \bigg( (p^n-1) \frac{(p^n+1)^{r-1} - 2^{r-1}}{(p^n-1)} +2^{r-1} \bigg) \\
& = p^m (p^n+1)^{r-1}  \delta_{T}^{(m)} = \lambda p^{m-1}\delta^{(m)}_T.
\end{align*}
From this, we deduce \eqref{Fs1-relation} in the case when $s_1$ is removable.

\paragraph{The case when $s_1$ is not removable}
In this case, $S$ is irreducible (i.e. none of its elements is removable), hence the situation is more symmetric. Therefore, we will not need to distinguish between the cases $1\leq d\leq r-1$ and $d=r$ as above. However, we need to change the definition of $\lambda, \lambda_1, \dots, \lambda_r$ as follows. For all $1\leq k \leq r$ (including $k=r)$, we set
\begin{equation}
    \lambda_k \colonequals 2^{k-1} p^{s_k} (p^n+1)^{r-k}.
\end{equation}
To define $\lambda$, we consider the polynomial
\begin{equation}
    Q(X)=\sum_{j=1}^{r-1} 2^{r-j-1} X^j - 2^{r-1} = \frac{X^r-2^r}{X-2}-2^r.
\end{equation}
It is clear that if $t > 2$ then $Q(t) > 0$. We put
\begin{equation}
    \lambda \colonequals Q(p^n+1) = \frac{(p^n+1)^r-2^r}{p^n-1}-2^r =\frac{(p^n+1)^r-2^r p^n}{p^n-1} .
\end{equation}
Again, we check that the identity \eqref{Fs1-relation} holds. Letting again $\Pi(x)$ denote the right-hand side of \eqref{Fs1-relation}, we have
\begin{align*}
  \Pi(x) &=  \sum_{k=1}^{r} 2^{k-1} p^{s_k} (p^n+1)^{r-k} F_{T_{k+1}}^{(s_k)}(x) \\
&=  \sum_{k=1}^{r} 2^{k-1} p^{s_k} (p^n+1)^{r-k}  \sum_{i=0}^{n-1} p^i \delta_{T_{k+1}}^{(i+s_k)} x_{i+s_k}.
\end{align*}
For $m\in [s_{d},s_{d+1}[$ and $1\leq d\leq r$, the coefficient in front of $x_m$ in $\Pi(x)$ is
\begin{align*}
c_m & = p^m \left( \sum_{k=1}^d 2^{k-1}(p^n+1)^{r-k} \delta_{T_{k+1}}^{m} + p^{n} \sum_{k=d+1}^{r} 2^{k-1} (p^n+1)^{r-k} \delta_{T_{k+1}}^{m} \right) \\
&= p^m \delta_{T}^{(m)}  \Bigg( \frac{(p^n+1)^{r}-2^{d-1}(p^n+1)^{r+1-d}}{p^n-1} -  2^{d-1}(p^n+1)^{r-d}  + \frac{2^d p^n}{p^n-1} ((p^n+1)^{r-d}-2^{r-d}) \Bigg)  \\
& = p^m \delta_{T}^{(m)}  \bigg( \frac{(p^n+1)^{r} - 2^r p^n}{p^n-1}  \bigg) = \lambda p^m \delta^{(m)}_T.
\end{align*}
This terminates the proof of \eqref{Fs1-relation} in the case when $s_1$ is not removable.

\subsubsection{\texorpdfstring{Proof of the inclusion $\Ccal_S\subset \Ccal^{\cap,+}_S$:}{}}

\begin{proposition} \label{gen-prop}
Let $t\in S$ be an element.
\begin{assertionlist}
    \item If $t\notin T$, then $\gen_{R,S}^{(t)}\in \Ccal_{\pha,S}$.
    \item  If $t\in T$, then $\gen_{R,S}^{(t)}\in \bigcap_{j\in S} \Ccal_{S\setminus\{j\}}$.
\end{assertionlist}
\end{proposition}

\begin{proof}
When $t\notin S$ then $\gen_{R,S}^{(t)}=\ha_{R,S}^{(t)}\in \Ccal_{\pha,S}$. To prove (2), write again $S=\{s_1,\dots, s_r\}$, where for all $i$, $s_{i+1}$ follows $s_i$ in $S$ (see section \ref{order-sec}). We write $t=s_k$ for some $1\leq k\leq r$ and assume that $t\in T$. For each $1\leq j \leq r$, we need to prove that the vector $x=\gen_{R,S}^{(t)}$ lies in $\Ccal_{S\setminus\{s_j\}}$, i.e. that it satisfies the inequality $F^{(s_d)}_{T_j}(x)\leq 0$ for all $d\neq j$, where $T_j=\Phi(S\setminus\{s_j\})$. Since $t\in T$, we have $x= - e_t-\delta_R^{(t-1)} p e_{t-1}$. 

First, assume that $s_d\neq t$. In this case, the coefficient of $x_{t-1}$ and $x_t$ in the $p$-expression $ F^{(s_d)}_{T_j}(x)$ are $\delta_{T_j}^{(t-1)} p^m$ and $\delta_{T_j}^{(t)} p^{m+1}$ respectively (for a certain integer $m$). Hence, for all $d\neq j$, $1\leq d \leq r$, the sign of $F^{(s_d)}_{T_j}(x)$ when $x=\gen_{R,S}^{(t)}$ is the same as the sign of the expression
\begin{equation}
    -\delta_{T_j}^{(t)} - \delta_{T_j}^{(t-1)} \delta_{R}^{(t-1)}.
\end{equation}
Thus, it suffices to show that not both of $\delta_{T_j}^{(t)}$ and $\delta_{T_j}^{(t-1)} \delta_{R}^{(t-1)}$ are equal to $-1$. Thus we may assume $t\notin T_j$. Since $t\in T$ by assumption, we are reduced to the case $j=k+1$ (in all other cases one would have $t\in T_j$). But in this case, we have $t\in S\setminus \{s_j\}$ so the admissibility condition for $\Ccal_{S\setminus\{s_j\}}$ implies that $\delta_{T_j}^{(t-1)} \delta_{R}^{(t-1)}=1$.

Finally, we consider the case when $s_d=t$. In particular we have $t\in S\setminus \{s_j\}$.
In this case, the coefficients of $x_{t-1}$ and $x_t$ in the $p$-expression $F^{(t)}_{T_j}(x)$ are $\delta_{T_j}^{(t-1)} p^{n-1}$ and $\delta_{T_j}^{(t)}$ respectively. Hence, the sign of $F^{(t)}_{T_j}(x)$ when $x=\gen_{R,S}^{(t)}$ is the same as the sign of
\begin{equation}
    -\delta_{T_j}^{(t)} - p^n \delta_{T_j}^{(t-1)} \delta_{R}^{(t-1)}.
\end{equation}
Again, since $t\in S\setminus \{s_j\}$ in this case we deduce by admissibility that $\delta_{T_j}^{(t-1)} \delta_{R}^{(t-1)}=1$, and the result follows.
\end{proof}

Finally, we complete the proof of the inclusion $\Ccal_S\subset \Ccal^{\cap,+}_S$. Recall that the kernel $\KK(\Ccal_S)$ coincides with $\KK(\Ccal_{\pha,S})$. Moreover, $\Ccal_S$ is $\QQ_{\geq 0}$-generated modulo kernel by the elements $\gen_{R,S}^{(t)}$ for $t\in S$. By Proposition \ref{gen-prop}, we have $\gen_{R,S}^{(t)}\in \Ccal^{\cap,+}_S$ for all $t\in S$. It follows that for any $x\in \Ccal_S$, there is an integer $N\geq 1$ such $Nx\in \Ccal^{\cap,+}_S$. By definition the cone $\Ccal^{\cap,+}_S$ is saturated, so we have $x\in \Ccal^{\cap,+}_S$. This terminates the proof of Theorem \ref{main-thm-hom}.

\subsection{Hasse-regularity}

Let $X=X_R$ be a Shimura variety of $A_1$-type and parabolic type $R$. Recall by \eqref{inclu-w} that for any subset $S\subset E_n$ we have inclusions
\begin{equation}
    \Ccal_{\pha,S} \ \subset \ \Ccal_{X,S} \ \subset \ \Ccal_S
\end{equation}
where we used the equality $\Ccal_S^{\cap,+}=\Ccal_S$ (Theorem \ref{main-thm-hom}) and we omitted the cone $\Ccal_{\flag, S}$ that we will not use in what follows.

\begin{definition}
We say that $S$ is Hasse-regular if $\Ccal_{\pha,S} \ = \ \Ccal_{X,S}$.
\end{definition}

In particular, if $\Ccal_{S} = \Ccal_{\pha,S}$ then $S$ is obviously Hasse-regular. The notion of Hasse-regularity can be defined for arbitrary Hodge-type Shimura varieties. In the case of unitary Shimura varieties of signature $(n-1,1)$ at a split prime, Goldring and the author have shown in \cite{Goldring-Koskivirta-GS-cone} that the stratum $\Flag(X)_w$ is Hasse-regular for all elements $w\in W$ in the closure of the element $w_{0,I}w_0$ (the longest element of ${}^I W$). In the theorem below, we give a characterization of the strata $S$ which satisfy $\Ccal_{S}=\Ccal_{\pha,S}$. When $|S|=1$ this equality holds by definition. Let $\Gamma$ be the chain diagram of parabolic type $R$ and stratum $S$. Write $S^c$ for the complement of $S$ in $E_n$. Write $\# \ F$ to denote the number of element of a finite set $F$.

\begin{theorem} \label{thm-hasse-reg}
Assume $|S|\geq 2$. Then the following are equivalent:
\begin{equivlist}
    \item One has $\Ccal_S = \Ccal_{\pha,S}$.
        \item $\Ccal_{\pha,S}$ is homogeneous.
    \item For any connected component $C\subset \Gamma$, the numbers $\# \ C\cap R$ and $\# \ C$ have different parity.
\end{equivlist}
\end{theorem}

\begin{proof}
Since $\Ccal_{\pha,S}$ is a positive, admissible, $S$-adapted $p$-cone, (i) and (ii) are obviously equivalent. Using the description of the function $\rho_{\pha,S}$ given in \eqref{rho-pha}, we see that $\rho_{\pha,S}$ is a constant function if and only if for any $j\in E_n$, the parity of the number
\begin{equation}
   \# \ [j,i-1[ \ \cap \ R \ \ + \ \ \# \  ]j,i-1[ \ \cap \ S^c 
\end{equation}
is independent of $i\in S$. If we change $i$ to the element $i^+$ that follows $i$ in $S$, the above number changes by the quantity $\# \ C\cap R \ + \ \# \ C\cap S^c$ where $C$ is the connected component of $i^+$. Since $i$ can be chosen arbitrarily in $S$ and $|S|\geq 2$, we deduce that condition (ii) is equivalent to $\# \ C\cap R \ \equiv \ \# \ C\cap S^c \pmod{2}$ for all connected components $C$ in $\Gamma$. Since each connected component $C$ has a unique element in $S$ by definition of $\Gamma$, one has $\# \ C\cap S^c \ = \ (\# \ C) -1$. The result follows.

%Using \eqref{CS-eq}, Condition (i) is also equivalent to the condition
%\begin{equation}
%    \gen^{(j)}_{S}\in \Ccal_{\pha,S}
%\end{equation}
%for all $j\in S$. Writing $T\colonequals \Phi(S)$, we have $\gen^{(j)}_{S}=\delta_T^{(j)} e_j-p\delta_R^{(j-1)} e_{j-1}$ for all $j\in E_n$. If $j\in S\setminus T$ we have $\gen^{(j)}_{S}=\ha^{(j)}_{S}$ so this weight always lies in $\Ccal_{\pha,S}$. Hence (i) is equivalent to $-e_j-p\delta_R^{(j-1)} e_{j-1}\in \Ccal_{\pha,S}$ for all $j\in S\cap T$. Consider the inequalities provided by Proposition \ref{prop-pha-cone}. From them, we see that $-e_j-p\delta_R^{(j-1)} e_{j-1}$ lies in $\Ccal_{\pha,S}$ if and only if for any $i\in S$, we have
\end{proof}

In particular, when the equivalent conditions of Theorem \ref{thm-hasse-reg} are satisfied, we have equalities
$\Ccal_{\pha,S} \ = \ \Ccal_{X,S} \ = \ \Ccal_S$. Therefore, one could say as a slogan that for those strata, "the weight of any nonzero mod $p$ automorphic form on $\overline{\Flag}(X)_S$ is spanned by the weights of partial Hasse invariants of the stratum".

\subsection{The general case}
So far we have only considered the case when $G$ is a Weyl restriction $G=\Res_{\FF_{p^n}/\FF_p}(\GL_{2,\FF_{p^n}})$. In the context of Shimura varieties of $A_1$-type explained in \ref{Shim-A1-type-sec}, this group corresponds to the case when $p$ is inert in the totally real field $\mathbf{F}$. If $p$ follows a more general ramification pattern, we need to consider a group of the form
\begin{equation}
    G=\Res_{\FF_{p^{m_1}}/\FF_p}(\GL_{2,\FF_{p^{m_1}}}) \times \dots \times \Res_{\FF_{p^{m_r}}/\FF_p}(\GL_{2,\FF_{p^{m_r}}})
\end{equation}
where $n=m_1+\dots + m_r$ is a partition of $n$ by positive integers. In this case, one can apply to each $\FF_p$-factor $G_i=\Res_{\FF_{p^{m_i}}/\FF_p}(\GL_{2,\FF_{p^{m_i}}})$ all of the theory of $p$-cones carried out in sections \ref{p-cone-sec} and \ref{coho-van-sec}. As we explained before, all group-theoretical objects decompose with respect to the $\FF_p$-factors. For any subset $S\subset E_n$, we may decompose $S=S_1\sqcup \dots \sqcup S_r$ for subsets $S_i$ corresponding to the $\FF_p$-factors and similarly $R=R_1\sqcup \dots \sqcup R_r$ for the parabolic type. Then, define $\Ccal_{R,S}$ as the direct product of the $\Ccal_{R_i,S_i}$. With this definition, Theorem \ref{vanish-thm} on the vanishing of cohomology holds without modification. Similarly, a version of Theorem \ref{thm-hasse-reg} can easily be formulated in this more general setting as well.

\subsection{The maximal stratum}

We return to the case $G=\Res_{\FF_{p^n}/\FF_p}(\GL_{2,\FF_{p^n}})$. We now restrict ourselves to the case of the maximal stratum, parametrized by the longest element $w_0\in W$. For groups of $A_1$-type, it corresponds to the subset $S=E_n$. In this case, note that we have
\begin{equation}
    \Phi_R(E_n)=R.
\end{equation}
By Theorem \ref{main-thm-hom}, we know that the cone $\Ccal_X$ is contained in the cone $\Ccal_{E_n}$, which is defined by the $n$ inequalities
\begin{equation}
    F^{(i)}_{R}(x)\leq 0, \quad x\in \ZZ^n, \quad i=1,\dots,n.
\end{equation}
The method used to prove this result does not use the fact that the map $\zeta_{\flag}\colon \Flag(X) \to \GF^\mu$ is a base extension of a map $\zeta\colon X\to \GZip^\mu$. Hence it would hold more generally for any smooth morphism $\zeta'\colon Y\to \GF^\mu$ satisfying certain conditions (replacing the stratum $\Flag(X)_w$ with the preimage $\zeta'^{-1}(\Flag_w)$). The fact that $\zeta_{\flag}$ is a base extension of $\zeta$ forces $\Ccal_X$ to be contained in the $L$-dominant cone $X^*_{+,I}(T)$, as we mentioned in \eqref{precise-inclu-w0}. Therefore, the cone $\Ccal_X$ satisfies the additional constraints
\begin{equation}
    x_i\geq 0 , \quad i\in R.
\end{equation}
Here we determine a set of generators for the intersection $\Ccal_{E_n} \cap X^*_{+,I}(T)$. Denote the elements of $E_n\setminus R$ as follows:
\begin{equation}
    E_n \setminus R=\{r_1,\dots, r_h\}
\end{equation}
where $h=n-|R|$. Again, we assume for each $i$ that $r_{i+1}$ is the element of $E_n\setminus R$ that follows $r_i$ (where $r_{h+1}=r_1$). Let $g_i$ denote the gap between $r_{i-1}$ and $r_i$, i.e. the smallest positive integer $g_i$ such that $r_{i-1}+g_i=r_i$. Let $(e_1,\dots , e_n)$ denote the canonical basis of $\ZZ^n$ and extend the definition of $e_i$ to any $i\in \ZZ$ by $n$-periodicity. For each $1\leq i \leq h$ and each $k=1,\dots , g_i-1$ (when $g_i\geq 2$), define the following elements
\begin{equation} \label{gen-zip-1}
    \lambda^{(i)}_{k} \colonequals e_{r_i} + p^k e_{r_i-k}
\end{equation}
Moreover, for $k=g_i$ define
\begin{equation}\label{gen-zip-2}
    \lambda^{(i)}_{g_i} \colonequals e_{r_i} - p^{g_i} e_{r_{i-1}}.
\end{equation}
We thus obtain exactly $n=\sum_{i=1}^h g_i$ weights of the form $\lambda^{(i)}_{k}$ for $1\leq i \leq h$ and $1\leq k \leq g_i$. It is immediate to check that the weights $\lambda^{(i)}_{k}$ are $L$-dominant and satisfy the inequalities $F^{(i)}_R(x)\leq 0$ for $i=1,\dots, n$, thus they lie in the intersection $\Ccal_{E_n} \cap X^*_{+,I}(T)$.

\begin{proposition}\label{prop-cone-X-dom}
The cone $\Ccal_{E_n} \cap X^*_{+,I}(T)$ is $\QQ_{\geq 0}$-generated by the weights $\lambda^{(i)}_{k}$ for $1\leq i \leq h$ and $1\leq k \leq g_i$.
\end{proposition}

It will be convenient to modify $\lambda^{(i)}_{k}$ by defining $\eta^{(i)}_{k}\colonequals \frac{1}{p^k} \lambda^{(i)}_{k}$. Before giving the proof, we start with the following easy lemma.

\begin{lemma}
The elements $\lambda^{(i)}_{k}$ (for $1\leq i \leq h$, $1\leq k \leq g_i$) form a basis of $\QQ^n$.
\end{lemma}

\begin{proof}
When $p$ goes to infinity, $\eta^{(i)}_{k}$ converges to $\pm e_{r_i-k}$, which are clearly linearly independent vectors. The result follows.
\end{proof}

We now prove Proposition \ref{prop-cone-X-dom} by showing that any $x=(x_1,\dots,x_n) \in \Ccal_{E_n} \cap X^*_{+,I}(T)$ is a linear combination with non-negative rational coefficients of the $\eta^{(i)}_{k}$. For each $1\leq j\leq n$, define $k_j$ as follows:
\begin{equation}
    k_j\colonequals \max \{ u\geq 1 \ | \ e_{j+u} \notin R \}.
\end{equation}
Consider the matrix $M$ whose $j$th column is the vector 
\begin{equation}
    \delta_R^{(j)} e_j + p^{-k_j} e_{j+k_j}. 
\end{equation}
Note that this vector is of the form $\eta^{(i)}_{k}$ for $i$ such that $e_i=e_{j+k_j}$ and $e_j=e_{r_i-k}$. We have simply changed the ordering of the vectors so that the diagonal coefficients of the matrix are $\pm 1$. Then, one checks easily that for $x\in \ZZ^n$, the vector $y=M^{-1} x$ satisfies the following: If $i\in R$, then the $i$-th coordinate of $y$ is 
\begin{equation}
   y_i = x_i.
\end{equation}
For $i\notin R$, the $i$th coordinate of $y$ is
\begin{equation}
  y_i = - \frac{1}{p^n-1} \sum_{u=0}^{n-1} \delta^{(i+1+u)}_R p^{u-1} x_{i+1+u} = - \frac{1}{p(p^n-1)} F^{(i+1)}_{R}(x).
\end{equation}
Since $x\in \Ccal_{E_n} \cap X^*_{+,I}(T)$, we obtain $y_i\geq 0$ for all $1\leq i \leq n$. This terminates the proof of Proposition \ref{prop-cone-X-dom}.

\subsection{The Cone Conjecture} \label{sec-zip-proof}

Recall that we defined in section \ref{low-cone-sec} the lowest weight cone $\Ccal_w$. We also explained that under a certain condition (satisfied for groups that are Weyl restrictions of split $\FF_p$-groups) the cone $\Ccal_{\lwsf}$ is contained in $\Ccal_{\zip}$. Let $X=X_R$ be a Hodge-type Shimura variety of $A_1$-type and parabolic type $R\subset E_n$. By the equality $\Ccal_{E_n}^{\cap,+} = \Ccal_{E_n}$ (Theorem \ref{main-thm-hom}), the inclusions \eqref{all-inclu-w0} and the containment $\Ccal_{\lwsf}\subset \Ccal_{\zip}$, we obtain:
\begin{equation}
    \Ccal_{\lwsf} \ \subset \ \Ccal_{\zip} \ \subset \ \Ccal_X \ \subset \ \Ccal_{E_n} \cap X^*_{+,I}(T).
\end{equation}
Recall that the Cone Conjecture \ref{conj-cones} predicts that we have $\Ccal_{\zip} = \Ccal_X$.

\begin{theorem}\label{thm-cone-conj}
When $X=X_R$ is a Hodge-type Shimura variety of $A_1$-type, Conjecture \ref{conj-cones} holds true. More precisely, one has 
\begin{equation}\label{equal-4}
    \Ccal_{\lwsf} \ = \ \Ccal_{\zip} \ = \ \Ccal_{X} \ = \ \Ccal_{E_n} \cap X^*_{+,I}(T).
\end{equation}
This cone is generated over $\QQ_{\geq 0}$ by the weights $\lambda_k^{(i)}$ defined in \eqref{gen-zip-1} and \eqref{gen-zip-2}.
\end{theorem}

\begin{proof}
It suffices to show that $\Ccal_{\lwsf} = \Ccal_{E_n} \cap X^*_{+,I}(T)$. Recall that for a general group $G$, the cone $\Ccal_{\lwsf}$ is the set of $\lambda\in X^*_{+,I}(T)$ satisfying the inequality \eqref{formula-norm-low} for each $\alpha\in \Delta^{P_0}$. Here, recall that $P_0$ is the smallest parabolic subgroup containing $B$ defined over $\FF_p$. In other words, $P_0=\bigcap_{i\in \ZZ} \sigma^i(P)$. In the trivial case when $R=E_n$, one has $P=G$ and one sees easily that $\Ccal_{\lw}=X^*_+(T)$, thus all four cones in \eqref{equal-4} coincide with $X^*_+(T)$. Note that in this case $X$ is a zero dimensional variety by \eqref{dimXR}, but $\Flag(X)$ is a $G/B$-fibration over $X$, which explains why $\Ccal_X=X^*_+(T)$.

In all other cases, we have
\begin{equation}
    P_0= \bigcap_{i\in \ZZ} \sigma^i(P_R) = \bigcap_{i\in \ZZ} P_{\sigma^i(R)} = B.
\end{equation}
Hence we have $W_{L_0}=\{1\}$, and the lowest weight cone is given by the set of $L$-dominant characters $x=(x_1,\dots, x_n)\in X^*_{+,I}(T)$ satisfying the inequalities
\begin{equation}
     \sum_{i=0}^{n-1} p^i \delta_R^{(i+j)} x_{i+j} \leq 0, \quad j=1,\dots, n.
\end{equation}
These are exactly the inequalities defining the cone $\Ccal_{E_n}$. This terminates the proof.
\end{proof}

\subsection{The Hilbert--Blumenthal case}
This section is dedicated to the special case of Hilbert--Blumenthal varieties (i.e. the case $R=\emptyset$). We make explicit in this case the main theorems proved for a general parabolic type $R\subset E_n$. In this case $\Flag(X)=X$ and the line bundle $\Lcal(\mathbf{k})$ coincides with the line bundle $\omega^{\mathbf{k}}$ defined in \eqref{omegak-intro}. For a subset $S\subset E_n$, the flag stratum $\Flag(X)_S$ is simply the Ekedahl--Oort stratum parametrized by $S$, that we denote by $X_S$. The function
\begin{equation}
    \Phi\colon \Pcal(E_n)\to \Pcal(E_n)
\end{equation}
can be concretely described as follows: For a subset $S\subset E_n$, let $\Gamma=\Gamma_n(\emptyset,S)$ be the chain diagram of type $\emptyset$ for the stratum $S$. For each element $s\in S$, write $\gamma(s)$ for the number of elements in the connected component of $\Gamma$ whose head is $s-1$. It is the smallest positive integer such that $s-\gamma(s)\in S$. Then, define $\Phi(S)$ as the set
\begin{equation}
    \Phi(S)=\{s-i \ | \ s\in S, \ i \textrm{ odd}, \ 1\leq i <\gamma(s) \}.
\end{equation}
We write $\Ccal_S=\Ccal_{\emptyset,S}$ for the unique positive admissible homogeneous $S$-adapted $p$-cone. It is given by the inequalities $F^{(i)}_{\Phi(S)}(x)\leq 0$ for $x\in \ZZ^n$ where $i$ varies in $S$. Theorem \ref{vanish-thm} then translates as follows:
\begin{theorem}\label{thm-vanish-HB}
Let $S\subset E_n$ be a subset and $\mathbf{k}\in \ZZ^n$. If $\mathbf{k}\notin \Ccal_S$, then
\begin{equation}
H^0(\overline{X}_S,\omega^{\mathbf{k}}) =0.
\end{equation}
\end{theorem}

Next, we make the link with previous results of Goldring and the author in \cite{Goldring-Koskivirta-global-sections-compositio}. In this previous work, we introduced a combinatorial notion of "admissibility" for strata (unrelated to the admissibility of $p$-cones defined in this paper). This condition says precisely (using the terminology of the present paper) that all connected components of $\Gamma$ contain an odd number of elements. For those strata, the authors proved that one has an equality
\begin{equation}
    \Ccal_{X,S} = \Ccal_{\pha,S}.
\end{equation}
This is simply a special case of Theorem \ref{thm-hasse-reg}. Indeed, for $R=\emptyset$ it translates as follows:

\begin{theorem}
Assume $|S|\geq 2$. Then the following are equivalent:
\begin{equivlist}
\item One has $\Ccal_S = \Ccal_{\pha,S}$.
\item $\Ccal_{\pha,S}$ is homogeneous.
\item For any connected component $C\subset \Gamma$, the number $\# \ C$ is odd.
\end{equivlist}
\end{theorem}

\subsection{Siegel-type Shimura varieties}

We end this article with a brief discussion about possible generalizations to other Shimura varieties. One of the main examples that we have in mind is the case of Siegel-type Shimura varieties $\Acal_n$, which parametrizes principally polarized abelian varieties of dimension $n$ (with a level structure). This variety is associated with the group $\GSp_{2n}$, which seems totally unrelated to the groups of $A_1$-type considered in this paper. However, the two cases bear a striking resemblance. In \cite{Goldring-Koskivirta-global-sections-compositio, Goldring-Koskivirta-divisibility}, Goldring and the author completely determined for $n=2,3$ the cone $\Ccal_K$ defined in \eqref{CK-def} in the case of Siegel-type Shimura varieties (when $n=3$, we only proved the result for $p\geq 5$). For $n=3$, we showed that $\Ccal_K$ is given by
\begin{equation}\label{ineq-Sp}
    \Ccal_K = \left\{ (x_1,x_2,x_3)\in X^*_{+,I}(T) \ \relmiddle| \ \parbox{4cm}{ $p^2 x_1 + x_2 + px_3 \leq 0 \\
    px_1 + p^2x_2 + x_3 \leq 0$ } \right\}.
\end{equation}
In this case, the condition $(x_1,x_2,x_3)\in X^*_{+,I}(T)$ simply means that $x_1\geq x_2\geq x_3$. In other words, $\Ccal_K$ is the intersection of a $p$-cone with $X^*_{+,I}(T)$ just as for the cases considered in this article (see Theorem \ref{thm-cone-conj}). Even better: One can add to the two inequalities in \eqref{ineq-Sp} a third one, given by $x_1 + px_2 + p^2x_3 \leq 0$. This does not change the right-hand side of \eqref{ineq-Sp} because this inequality is implied by the two other inequalities combined with the condition $(x_1,x_2,x_3)\in X^*_{+,I}(T)$. Thus, we see that the cone $\Ccal_K$ for the group $\GSp_6$ can be written as
\begin{equation}
   \Ccal_K = \Ccal_{E_3} \cap X^*_{+,I}(T).
\end{equation}
In other words, the inequalities defining $\Ccal_K$ inside $X^*_{+,I}(T)$ are the same as for Hilbert--Blumenthal Shimura varieties. The only difference is the $L$-dominance condition, which has a different meaning for the group $\GSp_3$.

\bibliographystyle{test}
\bibliography{biblio_overleaf}

\end{document}